\documentclass[12pt]{amsart}
\usepackage{amssymb}
\usepackage{amsmath,esint}
\usepackage{colortbl}
\usepackage[dvipsnames]{xcolor}
\usepackage[colorlinks, allcolors=RedViolet,pdfstartview=,pdfpagemode=UseNone]{hyperref} 
\usepackage{graphicx}
\usepackage{pgf,tikz}
\usetikzlibrary{arrows}
\usepackage{fullpage}

\newtheorem{theorem}[equation]{Theorem}%[section]
\newtheorem{lemma}[equation]{Lemma}
\newtheorem{proposition}[equation]{Proposition}

\theoremstyle{definition}
\newtheorem{definition}[equation]{Definition}
\newtheorem{assumption}[equation]{Assumption}

\theoremstyle{remark}
\newtheorem{remark}[equation]{Remark}
\newtheorem{example}[equation]{Example}
\numberwithin{equation}{section}

\newcommand{ \R }{ \mathbb{R} }
\newcommand{ \Rn }{ {\mathbb{R}^n} }
\newcommand{\Phiw}{{\Phi_{\mathrm{w}}}}
\newcommand{\Phic}{{\Phi_{\mathrm{c}}}}

\newcommand{\bphi}{{\bar\phi}}
\newcommand{\bA}{{\bar A}}
\newcommand{\bF}{{\bar F}}
\newcommand{\bu}{{\bar u}}
\newcommand{\bv}{{\bar v}}
\newcommand{\bL}{\bar L}
\newcommand{\tL}{\tilde L}
\newcommand{\intA}{{A^{(-1)}}}

\newcommand{\supp}{\operatorname{supp}}

\renewcommand{\div}{\mathrm{div}}
\newcommand{\loc}{\mathrm{loc}}

\renewcommand{\epsilon}{\varepsilon}
\renewcommand{\phi}{\varphi}
\renewcommand{\le}{\leqslant}
\renewcommand{\ge}{\geqslant}
\renewcommand{\leq}{\leqslant}
\renewcommand{\geq}{\geqslant}

\newcommand{\ainc}[1]{\hyperref[ainc]{{\normalfont(aInc){\ensuremath{_{#1}}}}}}
\newcommand{\adec}[1]{\hyperref[adec]{{\normalfont(aDec){\ensuremath{_{#1}}}}}}
\newcommand{\inc}[1]{\hyperref[inc]{{\normalfont(Inc){\ensuremath{_{#1}}}}}}
\newcommand{\dec}[1]{\hyperref[dec]{{\normalfont(Dec){\ensuremath{_{#1}}}}}}
\newcommand{\azero}{\hyperref[azero]{{\normalfont(A0)}}}
\newcommand{\aone}{\hyperref[aone]{{\normalfont(A1)}}}

\newcommand{\wVA}{\hyperref[wVA1]{{\normalfont(wVA1)}}}
\newcommand{\VA}{\hyperref[VA1]{{\normalfont(VA1)}}}

\newcommand{\aones}[1]{\hyperref[aones]{{\normalfont(A1-{\ensuremath{#1}})}}}

\begin{document}

\title[Regularity theory without Uhlenbeck structure]
{Regularity theory for non-autonomous partial differential equations without Uhlenbeck structure}

% Information for first author
\author{Peter H\"ast\"o}
 % Address of record for the research reported here

\address{Department of Mathematics and Statistics, FI-20014 University of Turku, Finland}
\email{peter.hasto@utu.fi}

\author{Jihoon Ok}
 % Address of record for the research reported here
\address{Department of Mathematics, Sogang University, Seoul 04107, Republic of Korea}
\email{jihoonok@sogang.ac.kr}

% \thanks will become a 1st page footnote.
\thanks{}

% General info

%\date{\today}

%\dedicatory{This paper is dedicated to our advisors.}


\begin{abstract}
We establish maximal local regularity results of weak solutions or local minimizers of
\[
\div A(x, Du)=0
\quad\text{and}\quad
\min_u \int_\Omega F(x,Du)\,dx,
\]
providing new ellipticity and continuity assumptions on $A$ or $F$ with general $(p,q)$-growth. 
Optimal regularity theory for the above non-autonomous problems is a long-standing issue; 
the classical approach by Giaquinta and Giusti involves assuming that 
the nonlinearity $F$ satisfies a structure condition. This means that 
the growth and ellipticity conditions depend on a given special function, such as $t^p$, $\phi(t)$, $t^{p(x)}$, $t^p+a(x)t^q$, and  not only $F$ but also the given function is
assumed to satisfy suitable continuity conditions. Hence these regularity conditions depend on given special functions. 

In this paper we study the problem without recourse to special function structure and without 
assuming Uhlenbeck structure. 
We introduce a new ellipticity condition using $A$ or $F$ only, which entails that the function is 
quasi-isotropic, i.e.\ it may depend on the direction, but only up to a multiplicative constant. Moreover, we formulate the continuity condition on $A$ or $F$ without specific structure and without direct restriction on the ratio $\frac qp$ of the parameters from the $(p,q)$-growth condition. 
We establish local $C^{1,\alpha}$-regularity for some $\alpha\in(0,1)$ and $C^{\alpha}$-regularity for any $\alpha\in(0,1)$ of weak solutions and local minimizers.
Previously known, essentially optimal, regularity results are included as special cases.
\end{abstract}

%\marginpar{\small We dealt with minimizers with Uhlenbeck structure, $F(x,\xi)=F(x,|\xi|)$, in an earlier paper [\emph{Maximal regularity for local minimizers of non-autonomous functionals}, 
%J. Eur. Math. Soc., to appear].  In the present paper, we consider the general case.
%(We can consider including this when the paper is published. Jihoon: Ok)}

\maketitle

%%%%%%%%%%%%%%%%%%%%%%%%%%%%%%%%%%%%%%%%%%%%%%%%%%%%%%%%%%%%%%%%%%%%%%%%
%%%%%%%%%%%%%%%%%%%%%%%%%%%%%%%%%%%%%%%%%%%%%%%%%%%%%%%%%%%%%%%%%%%%%%%%
%%%%%%%%%%%%%%%%%%%%%%%%%%%%%%%%%%%%%%%%%%%%%%%%%%%%%%%%%%%%%%%%%%%%%%%%

\section{Introduction}\label{sect:intro}

Research on regularity of weak solutions or minimizers of the problems 
\[
\div A(x, Du)=0
\quad\text{and}\quad
\min_u \int_\Omega F(x,Du)\,dx
\]
is a major topic in the partial differential equations and the calculus of variations.
If there is no direct dependence on $x$ (i.e.,\ $A(x,\xi)\equiv A(\xi)$ and $F(x,\xi)\equiv F(\xi)$), 
these are called \textit{autonomous} problems. 
The simplest non-linear model cases is the $p$-power function
$F(\xi)=|\xi|^p$, $1<p<\infty$, and 
the corresponding Euler--Lagrange equation is the $p$-Laplace equation
where $A(Du)=|Du|^{p-2}Du$.
The maximal regularity of weak solutions of the $p$-Laplace equation is $C^{1,\alpha}$ for some $\alpha\in(0,1)$ depending only on $p$ and the dimension $n$ (e.g.,\ \cite{DiBe1,Eva3,Le1,Ura1}).

For \textit{non-autonomous} problems, there is also direct $x$-dependence. 
To tackle this case, Giaquinta and Giusti \cite{GiaG83,GiaG84} introduced the 
following $p$-type growth conditions:
\begin{equation}\label{p-growth}
\begin{cases}
\xi\mapsto F(x,\xi)\text{ is } C^2, \\
\nu |\xi|^p \le F(x,\xi)\le L(1+|\xi|^p), \\
\nu (\mu^2+|\xi|^2)^{\frac{p-2}{2}}|\lambda|^2 
\le D_\xi^2F(x,\xi)\lambda\cdot \lambda 
\le L (\mu^2+|\xi|^2)^{\frac{p-2}{2}} |\lambda|^2, \\
|F(x,\xi)-F(y,\xi)| \le \omega(|x-y|)(1+|\xi|^p). 
\end{cases} 
\end{equation}
Here, $F$ is related to the perturbed case $a(x)|\xi|^p$
and has the same $p$-type growth at all points. 
Lieberman \cite{Lie1} generalized the model by replacing $|\xi|^p$ 
with Orlicz growth $\phi(|\xi|)$. Marcellini \cite{Mar89} introduced non-standard, so-called  \textit{$(p,q)$-growth}
where the exponent $p$ on the right-hand 
side is replaced by $q>p$. In this situation, we need to assume that $\frac qp$ is close 
to $1$, see, e.g., \cite{BecMin20,BelSch20,DeFM21,Mar91}.
However, all these structure conditions fail to accommodate many kinds of energy functionals 
since the variability in the $x$- and $\xi$-direction are treated separately. 

For many years, 
it was thought that the only way to treat the $x$- and $\xi$-directions 
together was through special cases.
Consequently, a plethora of studies deal with the 
variable exponent case $F(x,\xi)=|\xi|^{p(x)}$.
Over the past half-dozen years the double phase functional $F(x,\xi)=|\xi|^p + a(x) |\xi|^q$, 
$1<p\leq q$ and $a\ge 0$, has attracted much attention. 
These models were first studied by Zhikov \cite{Zhi86, Zhi95} in the 1980's in relation to 
Lavrentiev's phenomenon and have been considered in thousands of papers \cite{Min06, Rad15}.
Moreover, various variants and borderline cases have been investigated, such as:
perturbed variable exponent $|\xi|^{p(x)} \log(e+|\xi|)$;
Orlicz variable exponent $\psi(|\xi|)^{p(x)}$ or $\psi(|\xi|^{p(x)})$;
degenerate double phase $|\xi|^p + a(x) |\xi|^p \log (e+|\xi|)$;
Orlicz double phase $\phi(|\xi|) + a(x) \psi(|\xi|)$;
triple phase $|\xi|^p + a(x) |\xi|^q + b(x) |\xi|^r$;
double variable exponent $|\xi|^{p(x)}+|\xi|^{q(x)}$; and
variable exponent double phase $|\xi|^{p(x)} + a(x) |\xi|^{q(x)}$. See \cite{HasO22} for 
references. We emphasize that all these special cases are covered by our results.

In \cite{HasO22}, we introduced a different approach which does not impose any direct restriction on 
$\frac qp$. However, we were only able to prove maximal regularity 
for local minimizers when $F(x,\xi)=F(x,|\xi|)$ has so-called \textit{Uhlenbeck structure}. In this 
article we extend the results to both minimizers and weak solutions and dispense with the 
Uhlenbeck restriction. 

\medskip

We collect some conditions for $A:\Omega\times\R^n\to \R^n$ and 
$F:\Omega\times\R^n\to \R$ with an open set $\Omega\subset\R^n$ ($n\ge 2$), which determine our equation and 
minimization energy, respectively. 
See Section~\ref{sect:preliminaries} for further definitions and notation, including 
%the other main assumption \wVA{}. 
the continuity assumption \wVA{} that is the other main assumption.

\begin{assumption}\label{ass:A}
We say that $A:\Omega\times\R^n\to \R^n$ satisfies Assumption~\ref{ass:A} if the following three conditions hold:
\begin{enumerate}\addtolength{\itemsep}{4pt}
\item%\label{hpzero} 
For every $x\in \Omega$, $A(x, 0)= 0$, $A(x,\cdot)\in C^{1}(\R^n\setminus\{ 0\}, \R^n)$
and for every $\xi\in \R^n$, $A(\cdot,\xi )$ is measurable.
\item%\label{hpone} 
There exist $L\ge 1$ and $1<p<q$ such that the radial function 
$t\mapsto | D_\xi A(x,te)|$ satisfies \azero{}, \ainc{p-2} and \adec{q-2} with the constant $L$,
for every $x\in \Omega$ and $e\in \Rn$ with $|e|=1$.  (The \textit{$(p,q)$-growth} condition) 
\item%\label{hptwo}
There exists $L\ge 1$ such that 
\[
| D_\xi A(x,\xi ')|
\le
L\, D_\xi A(x,\xi )e \cdot e 
\]
for all $x\in\Omega$ and $\xi ,\xi',e\in \Rn$ with $|\xi |=|\xi '|\neq 0$ and $|e|=1$. 
(The \textit{quasi-isotropic ellipticity} condition) 
\end{enumerate}
\end{assumption}

 The \azero{} condition in (2) means that a coefficient factor of $A$ is nondegenerate and nonsingular, for instance 
$a\approx 1$ when $A(x,\xi)=a(x) |\xi|^{p-2}\xi$. The \adec{q-2} and \ainc{p-2} conditions  in (2) are equivalent to the function $t\mapsto t^2| D_\xi A(x,te)|$ satisfying the $\Delta_2$- and $\nabla_2$-conditions, respectively. In particular, we note from (2) that $D_\xi A(x,\xi)\neq 0$ when $\xi\neq 0$.
Uhlenbeck structure has been replaced by (3), which is a quasi-isotropy condition 
since different directions behave the same up to a constant. It is known that 
completely anisotropic equations do not necessarily have any regularity as 
solutions may even be locally unbounded \cite{Gia87}. 
We also note that if $A(x,\xi)=D_\xi F(x,\xi)$  the condition (3) means that the Hessian matrix $D_\xi^2 F(x,\xi)$ with $\xi\neq0$ is positive definite and all its eigenvalues on each sphere for $\xi$ are comparable uniformly in $x$ and the radii of spheres, that is,
\[
1\le \frac{\sup\{\text{eigenvalues of }D_\xi^2 F(x,te) : |e|=1\}}
{\inf\{\text{eigenvalues of }D_\xi^2 F(x,te) : |e|=1\}} \le \tilde L \quad \text{for each }\ x\in\Omega\ \text{ and }\ t>0,
\]
where $\tilde L$ depends only on $L$ and $n$. Compare this to the $p$-growth condition in \eqref{p-growth}, 
which implies a stronger condition, where $x$ is inside the supremum and infimum: 
\[
1\le \frac{\sup\{\text{eigenvalues of }D_\xi^2 F(x,te) : x\in\Omega,\, |e|=1\}}
{\inf\{\text{eigenvalues of }D_\xi^2 F(x,te) : x\in\Omega,\ |e|=1\}} \le \tilde L \quad \text{for each }\  t>0.
\]
We further refer to \cite{DeFM21,DeFM22} for a discussion of non-uniformly elliptic conditions in terms of ratios of eigenvalues (in particular \cite[Section~4.6]{DeFM21} and \cite[Section~1]{DeFM22}) and related regularity results (see also Remark~\ref{rem:pqcondition}).

With these assumptions we consider the following quasilinear elliptic equation in divergence form: 
\begin{equation}\label{mainPDE}
\tag{{\(\div A\)}}
\div A(x,Du) = 0 \quad\text{in }\ \Omega.
\end{equation}
We say that $u\in W^{1,1}_{\loc}(\Omega)$  is a local \textit{weak solution} if $|Du|\,|A(\cdot,Du)|\in L^1_{\loc}(\Omega)$ and
\begin{equation*}%\label{weakform}
 \int_{\Omega} A(x,Du)\cdot D\zeta \,dx = 0\quad \text{for all }\ \zeta\in C^\infty_0(\Omega). 
\end{equation*}
We show that such solutions are quasiminimizers of a related functional with Uhlenbeck structure; 
thus $C^{0,\alpha}$-regularity for some $\alpha\in(0,1)$ and higher integrability 
follow from the results in \cite{HarHL21,HarHK18}, see Theorem~\ref{thm:holder}. 
Let us state the first main maximal regularity theorem, for general weak solutions.

\begin{theorem}[Maximal regularity for solutions]\label{thm:PDE}
Let $A:\Omega\times \Rn \to \Rn$ satisfy Assumption~\ref{ass:A} 
%and set $\intA(x,\xi ):=|\xi |\,A(x,\xi )$. 
and let $u\in W^{1,1}_{\loc}(\Omega)$ be a local weak solution to \eqref{mainPDE}. 
Define $A^{(-1)}(x,\xi):=|\xi|A(x,\xi)$. 
\begin{enumerate}
\item 
If $A^{(-1)}$ satisfies \wVA{}, then $u\in C^{0,\alpha}_{\loc}(\Omega)$ for every $\alpha\in(0,1)$.
\item 
If $A^{(-1)}$ satisfies \wVA{} with $\omega_\epsilon(r)\le r^{\beta_\epsilon}$ 
for some $\beta_\epsilon>0$, then $u\in C^{1,\alpha}_{\loc}(\Omega)$ for some $\alpha=\alpha(n,p,q,L,
\bL,\beta_\epsilon)\in(0,1)$.
\end{enumerate}
\end{theorem}

\begin{remark} 
The  continuity condition \wVA{} for $\Phi$-functions was introduced in \cite{HasO22}, see also Section~\ref{subsect:newcondition}. In our knowledge, the above theorem covers all previous known $C^{1,\alpha}$-regularity results for equation \eqref{mainPDE} with $\alpha$ independent of the solution. 
Some examples of functionals satisfying Assumptions~\ref{ass:A} or \ref{ass:F} are 
presented in Section~\ref{sect:example}.
\end{remark}

\begin{remark}
In \wVA{}, $\epsilon>0$ is arbitrary and it is possible 
that $\beta_\epsilon\to 0$ as $\epsilon \to 0$. Moreover, the constant $\bL=\bL_K$ is from \wVA{} and $K>0$ is  arbitrary.
However, for given structure we actually choose certain, positive $\epsilon$ and $K$
(see Section~\ref{sect:regularity}, in particular, \eqref{Kchoice} and \eqref{epsilon}), and $\alpha$ in the previous theorem depends on 
this $\beta_\epsilon$ and $\bL=\bL_K$. The same applies in Theorem~\ref{thm:functional}.
\end{remark}

\smallskip

If equation \eqref{mainPDE} is an Euler--Lagrange equation, that is, if $A=D_\xi F$ 
for some a function $F:\Omega\times \R^n\to\R$, then we can consider milder regularity assumptions 
in the context of variational calculus.

\begin{assumption}\label{ass:F}
We say that $F:\Omega\times\R^n\to [0,\infty)$ satisfies Assumption~\ref{ass:F} if the following two conditions hold:
\begin{enumerate}\addtolength{\itemsep}{4pt}
\item%\label{hfzero} 
For every $x\in \Omega$, $F(x, 0)=|D_\xi F(x,0)|= 0$, 
$F(x,\cdot)\in C^{1}(\R^n)\cap C^{2}(\R^n\setminus\{ 0\})$
and for every $\xi\in \R^n$, $F(\cdot,\xi)$ is measurable.
\item
The derivative $A:=D_\xi F$ satisfies conditions (2) and (3) of Assumption~\ref{ass:A}.
%\item%\label{hfone}
%There exist constants $L\ge 1$ and $1<p<q$ such that 
%$t\mapsto |D_\xi^2 F(x,te)|$ satisfies \azero{}, \ainc{p-2} and \adec{q-2} 
%for every $x,e\in \Rn$ with $|e|=1$. 
%\item%\label{hftwo} 
%The inequality
%\[
%|D_\xi^2 F(x,\xi')|
%\le
%L\, D_\xi^2 F(x,\xi)e \cdot e 
%\]
%with constant $L\ge1$ holds for all $\xi,\xi',e\in \Rn$ with $|\xi|=|\xi'|$ and $|e|=1$. 
\end{enumerate}
\end{assumption}

From this assumption it follows that $F(x,\xi)>0$ for all $x\in \Omega$ and 
$\xi\in \Rn\setminus\{0\}$, see \eqref{phifequiv} below.
%We say that 
%$u\in W^{1,F}_{\loc}(\Omega)$ 
%is a \emph{local minimizer} if 
%%$\int_{\Omega'} F(x,Du)\, dx<\infty$ for every $\Omega'\Subset\Omega$, and
%\begin{equation}\label{mainfunctional}
%\tag{{\(\min F\)}}
%\int_{\Omega'} F(x,Du)\, dx \le \int_{\Omega'} F(x,Dv)\, dx
%\quad\text{for every }v\in u+W^{1,F}_0(\Omega') \text{ and } \Omega'\Subset\Omega. 
%\end{equation}
We say that 
$u\in W^{1,1}_{\loc}(\Omega)$ 
is a \emph{local minimizer} if $F(\cdot,Du)\in L^1_{\loc}(\Omega)$ and
\begin{equation}\label{mainfunctional}
\tag{{\(\min F\)}}
\int_{\supp(u-v)} F(x,Du)\, dx \le \int_{\supp(u-v)} F(x,Dv)\, dx
\end{equation}
for every $v\in W^{1,1}_{\loc}(\Omega)$ with $\supp(u-v)\Subset\Omega$.

\begin{theorem}[Maximal regularity for minimizers]\label{thm:functional}
Let $F:\Omega\times \Rn \to [0,\infty)$ satisfy Assumption~\ref{ass:F} 
and let $u\in W^{1,1}_{\loc}(\Omega)$ be a local minimizer of \eqref{mainfunctional}. 
\begin{itemize}
\item[(1)] If $F$ satisfies \wVA{}, then $u\in C^{0,\alpha}_{\loc}(\Omega)$ for every $\alpha\in(0,1)$.
\item[(2)]  If $F$ satisfies \wVA{} with $\omega(r)\le r^{\beta_\epsilon}$ for some 
$\beta_\epsilon>0$, then $u\in C^{1,\alpha}_{\loc}(\Omega)$ for some 
$\alpha=\alpha(n,p,q,L,\bL,\beta_\epsilon)\in(0,1)$.
\end{itemize}
\end{theorem}

Under our differentiability assumptions on $F$, $u\in W^{1,1}_{\loc}(\Omega)$ 
is a local weak solution to \eqref{mainPDE} with $A= D_\xi F$
if and only if it is a local minimizer of \eqref{mainfunctional}.  
Since not every mapping $A$ is of the form $ D_\xi F$, Theorem~\ref{thm:PDE} is more 
general in terms of structure. On the other hand,  
\wVA{} with $G(x,\xi)=|\xi| D_\xi F(x,\xi)$ implies \wVA{} with $G(x,\xi)= F(x,\xi)$,
but not the other way around (cf.\ Proposition~\ref{prop:hphf}), 
so in that sense the assumption of Theorem~\ref{thm:functional} is weaker.

\begin{remark}\label{rem:pqcondition}
De Filippis and Mingione \cite{DeFM22} study H\"older regularity of the gradient of minimizers of 
non-autonomous, $(p, q)$-growth functionals with an upper bound of $\frac qp$. Their condition 
is not covered by the condition in the above theorem, but their H\"older exponent 
may depend on minimizers.
\end{remark}

Let us briefly outline the rest of the paper and point out the main novelties. We first 
collect some background information in next section. In Section~\ref{subsect:newcondition}, 
we introduce our new conditions including \wVA{} that have been adapted to the 
non-Uhlenbeck case. Formulating appropriate conditions and unifying them to cover all 
the cases was the first challenge that we faced. 
In Section~\ref{sect:constructionphi}, we construct 
for $A$ or $F$ an approximating function $\phi$ with Uhlenbeck structure (e.g., $F(x,\xi)\approx \phi(x,|\xi|)$ in the functional case) that we call a \textit{growth function}. The function 
$\phi$ is similar to the one used in \cite{HasO22} which allows us to use some earlier results. 
However, $\phi$ does not have the same continuity property \wVA{} as $A$ or $F$. 

In Section~\ref{sect:auxiliary}, we consider regularity results in 
two simpler cases that are used as building 
blocks later on. Specifically, we show that our weak solution or minimizer is also a quasiminimizer 
of a non-autonomous functional with Uhlenbeck structure  
and we study 
related $\bA$-equations or $\bF$-energy functionals which are autonomous but lack Uhlenbeck structure. 
The main difficulty for 
us was constructing an appropriate approximating autonomous problem with autonomous functions $\bA$ or $\bF$ from $A$ or $F$ and a growth function $\bphi$. This is solved in 
Section~\ref{sect:approx}. 
For the function $\phi$, we use the same approximation $\bphi$
as in \cite{HasO22}. However, for $A$ and $F$ we need a novel approach of transitioning 
smoothly between different growth using appropriately chosen functions $\eta_i$. 
With these elements in place, we prove the main results in Section~\ref{sect:regularity} 
using comparison and iteration arguments. In the final section, we apply the main result 
in the special cases of variable exponent- and double phase-type growth.

%%%%%%%%%%%%%%%%%%%%%%%%%%%%%%%%%%%%%%%%%%%%%%%%%%%%%%%%%%%%
%%%%%%%%%%%%%%%%%%%%%%%%%%%%%%%%%%%%%%%%%%%%%%%%%%%%%%%%%%%%
%%%%%%%%%%%%%%%%%%%%%%%%%%%%%%%%%%%%%%%%%%%%%%%%%%%%%%%%%%%%
\section{Preliminaries}\label{sect:preliminaries}

%%%%%%%%%%%%%%%%%%%%%%%%%%%%%%%%%%%%%%%%%%%%%%%%%%%%%%%%%%%%
\subsection{Notation and definitions}
Throughout the paper we always assume that $\Omega$ is a bounded domain in $\Rn$. Let $x_0\in \Rn$ and $r>0$. Then $B_r(x_0)$ is the standard open ball in $\Rn$ centered at $x_0$ with radius $r$. If its center is clear, we simply write $B_r=B_r(x_0)$.
A function $f:[0,\infty)\to [0,\infty)$ is \textit{almost increasing} or 
\textit{almost decreasing} if there exists $L\ge 1$ such that $f(t)\leq Lf(s)$ or $f(s)\leq Lf(t)$, respectively, for all $0<t<s<\infty$. In particular, if $L=1$, we say $f$ is 
\textit{increasing} or \textit{decreasing}.
 For an integrable function $f$ in $U\subset \Rn$, we define $\fint_U f\, dx := \frac{1}{|U|}\int_U f \,dx$ as the average of $f$ in $U$ in the integral sense.
For functions $f,g:U\to \R$, 
$f\lesssim g$ or $f\approx g$ (in $U$) mean that there exists $C\geq 1$ such that $f(y)\leq C g(y)$ or $C^{-1} f(y)\leq g(y)\leq C f(y)$, respectively, for all $y\in U$. 
By $D=D_x$ we denote the derivative with respect to the space-variable $x$.

For $\phi:\Omega\times [0,\infty) \to[0,\infty)$ and $B_r\subset \Rn$, we write
\[
\phi^+_{B_r}(z):=\sup_{x\in B_r\cap \Omega}\phi(x,z)
\quad \text{and}\quad
\phi^-_{B_r}(z):=\inf_{x\in B_r\cap \Omega}\phi(x,z).
\]
The same idea and notation is also used for $F:\Omega\times \R^N\to [0,\infty]$. 
If the map $t\mapsto\phi(x,t)$, $t\ge 0$, is increasing for every $x\in\Omega$, then the (left-continuous) 
inverse function with respect to $t$ is defined by
\[
\phi^{-1}(x,t):= \inf\{\tau\geq 0: \phi(x,\tau)\geq t\}. 
\] 
If $\phi$ is strictly increasing and continuous in $t$, then this is just the normal inverse function.

\begin{definition} \label{def:ainc}
We define some conditions for $\phi:\Omega\times[0,\infty)\to [0,\infty)$ and $\gamma\in\R$ related to regularity with 
respect to the second variable, which are supposed to hold for all $x\in \Omega$ and 
a constant $L\ge 1$ independent of $x$.  
\vspace{0.2cm}
\begin{itemize}
\item[\normalfont(aInc)$_\gamma$]\label{ainc} 
$t\mapsto \phi(x,t)/t^\gamma$ is almost increasing on $(0,\infty)$ with constant $L\geq 1$.
\vspace{0.2cm}
\item[\normalfont(Inc)$_\gamma$]\label{inc} 
$t\mapsto \phi(x,t)/t^\gamma$ is increasing on $(0,\infty)$. 
\vspace{0.2cm}
\item[\normalfont(aDec)$_\gamma$]\label{adec} 
$t\mapsto \phi(x,t)/t^\gamma$ is almost decreasing on $(0,\infty)$ with constant $L\geq 1$.
\vspace{0.2cm}
\item[\normalfont(Dec)$_\gamma$]\label{dec} 
$t\mapsto \phi(x,t)/t^\gamma$ is decreasing on $(0,\infty)$.
\vspace{0.2cm}
\item[\normalfont(A0)] \label{azero} $L^{-1}\leq \phi(x,1)\leq L$.
\end{itemize}
\end{definition}

Note that this version of \azero{} is slightly stronger than the one used in \cite{HarH19}, 
but they are equivalent under the doubling assumption \adec{}, which means that \adec{q} holds 
for some $q\ge 1$. 
We may rewrite \ainc{p} or \adec{q} ($p,q>0$) with constant $L\geq 1$, as 
\[
\phi(x,ct)\le Lc^p\phi(x,t)\quad\text{and}\quad L^{-1}C^p \phi(x,t)\le \phi(x,Ct) 
\]
\[
\text{or}\quad L^{-1}c^q \phi(x,t) \le \phi(x,ct)
\quad\text{and}\quad 
 \phi(x,Ct) \le L C^q \phi(x,t), \quad \text{respectively},
\]
for all $(x,t)\in \Omega\times (0,\infty)$ and $0<c\le 1 \le C$. We will use the above inequalities many times without mention.  
Moreover,
if $\phi(x,\cdot)\in C^1([0,\infty))$ for each $x\in\Omega$, then for $0<p \le q$
\[
\text{ $\phi$ satisfies \inc{p} and \dec{q}}
\quad \Longleftrightarrow \quad 
\text{$p \le \frac{t\phi'(x,t)}{\phi(x,t)} \le q \ $ for all $x\in\Omega$ and $t\in(0,\infty)$.}
\]

\medskip

We next introduce classes of $\Phi$-functions and generalized Orlicz spaces. 
We refer to \cite{HarH19} for more details about the basics. 
 In the sequel we omit the words ``generalized'' and ``weak'' 
mentioned in parentheses. 

\begin{definition}\label{defPhi}
Let $\phi:\Omega\times[0,\infty)\to [0,\infty]$. We call 
$\phi$ a \textit{(generalized) $\Phi$-prefunction} if $x\mapsto \phi(x,|f(x)|)$ is measurable for every measurable function $f$ on $\Omega$, and $t\mapsto \phi(x,t)$ is increasing for 
every $x\in\Omega$ and satisfies that 
$\phi(x,0)=\lim_{t\to0^+}\phi(x,t)=0$ and $\lim_{t\to\infty}\phi(x,t)=\infty$
for every $x\in\Omega$. A $\Phi$-prefunction $\phi$ is a 
\begin{itemize}
\item[(1)] \textit{(generalized weak) $\Phi$-function}, denoted $\phi\in\Phiw(\Omega)$, if it satisfies \ainc{1};
\item[(2)] \textit{(generalized) convex $\Phi$-function}, denoted $\phi\in\Phic(\Omega)$, if $t\mapsto \phi(x,t)$ is left-con\-tin\-u\-ous and convex for every $x\in\Omega$.
%\item[(3)] $\phi$ is an \textit{(generalized) $N$-function} if the map $t\mapsto\phi(x,t)$ is positive when $t>0$, convex, continuous and for every $x\in\Omega$, and $\lim_{t\to0}\frac{\phi(x,t)}{t}=0$ and $\lim_{t\to\infty}\frac{\phi(x,t)}{t}=\infty$.
\end{itemize}
If $\phi$ is independent of $x$, then we write $\phi\in\Phiw$ 
or $\phi\in\Phic$ without ``$(\Omega)$''. 
\end{definition}

We note that convexity implies \inc{1} hence
$\Phic(\Omega)\subset \Phiw(\Omega)$. While we are mainly interested in convex functions 
for minimization problems and related PDEs, the class $\Phiw(\Omega)$ is very useful for constructing 
approximating functionals, as it allows much more flexibility. This will be utilized several 
times in this article.

\medskip

We state some properties of $\Phi$-functions, for which we refer to 
\cite[Chapter 2]{HarH19}. We note that in \cite{HarH19}, $\phi\simeq \psi$ for $\Phi$-prefunctions $\phi$ and $\psi$ means that there exists $C\geq 1$ such that $\phi(x,C^{-1}t)\leq \psi(x,t)\leq \phi(x,Ct)$ for all $x\in \Omega$ and $t\in[0,\infty)$. This is weaker than $\approx$. However, if $\phi$ and $\psi$ satisfy \adec{}, 
then $\simeq$ and $\approx$ are equivalent.

 Suppose $\phi,\psi : [0,\infty)\to [0,\infty)$ are $\Phi$-prefunctions, $\phi$ satisfies \ainc{1}, and $\psi$ 
satisfies \adec{1}. Then there exist a convex $\Phi$-prefunction $\tilde\phi$ and a concave $\Phi$-prefunction $\tilde\psi$ such that $\phi\approx\tilde \phi$ and $\psi\approx\tilde \psi$. 
Therefore, by Jensen's inequality we have that 
\begin{equation*}%\label{Jensen}
\phi\left(\fint_\Omega |f|\,dx\right) \lesssim \fint_\Omega \phi(|f|)\,dx 
\quad\text{and}\quad 
\fint_\Omega \psi(|f|)\,dx \lesssim \psi\left(\fint_\Omega |f|\,dx\right)
\end{equation*}
for every $f\in L^1(\Omega)$, 
where the implicit constants depend on the constants from the equivalence relations and $L$ from \ainc{1} or \adec{1}.

We define the conjugate function of $\phi:[0,\infty)\to [0,\infty)$ by
$$
\phi^*(x,t) :=\sup_{s\geq 0} \, (st-\phi(x,s)).
$$ 
The definition directly implies  
$$
ts\leq \phi(x,t)+\phi^*(x, s)\quad \text{for all }\ s,t\ge 0.
$$
Furthermore, if $\phi$ satisfies \ainc{p} and \adec{q} for some $1<p\le q$, then $\phi^*$ satisfies \ainc{q'} and \adec{p'} with $p'=\frac{p}{p-1}$ and $q'=\frac{q}{p-1}$, and  for any 
$s,t\ge 0$ and $\epsilon\in(0,1)$,
\[
ts 
\leq 
\phi(x,\epsilon^{\frac{1}{p}}t)+\phi^*(x,\epsilon^{-\frac{1}{p}}s) 
\lesssim 
\epsilon \phi(x,t)+\epsilon^{-\frac{1}{p-1}} \phi^*(x,s)
\]
and
\[
ts 
\leq 
\phi(x,\epsilon^{-\frac{1}{q'}}t)+\phi^*(x,\epsilon^{\frac{1}{q'}}s) 
\lesssim 
\epsilon^{-(q-1)} \phi(x,t)+\epsilon \phi^*(x,s).
\]
We will refer to any of the previous three formulas as ``\textit{Young's inequality}''. We also note that  $(\phi^*)^*=\phi$ if $\phi\in \Phic(\Omega)$ by \cite[Theorem~2.2.6]{DieHHR11}. 

If $\phi\in\Phic(\Omega)$, then there exists $\phi'$, which is increasing and right-continuous, with
$$
\phi(x,t)=\int_0^t \phi'(x,s)\, ds.
$$
We collect some results about this (right-)derivative $\phi'$. 

\begin{proposition}[Proposition~3.6, \cite{HasO22}]\label{prop0} 
Let $\gamma>0$ and suppose that $\phi\in\Phic(\Omega)$ has derivative $\phi'$.
\begin{itemize}
\item[(1)] 
If the derivative $\phi'$ satisfies \ainc{\gamma}, \adec{\gamma}, \inc{\gamma} or \dec{\gamma}, then the function $\phi$ satisfies \ainc{\gamma+1}, \adec{\gamma+1}, \inc{\gamma+1} or \dec{\gamma+1}, respectively, with the same constant $L\geq 1$. 
\item[(2)] 
If $\phi$ satisfies \adec{\gamma}, then $\phi(x,t)\approx t \phi'(x,t)$.
\item[(3)] 
If $\phi'$ satisfies \azero{} and \adec{\gamma} with constant $L\geq 1$, then $\phi$ also satisfies \azero{}, with constant depending on $L$ and $\gamma$.
\item[(4)] 
$\phi^*(x,\phi'(x,t))\le t\phi'(x,t)$. 
\end{itemize}
\end{proposition}

We will use the following inequality for $C^1$-regular $\Phi$-functions. 

\begin{proposition}[Proposition~3.8, \cite{HasO22}]\label{prop000}
Let $\phi\in \Phic\cap C^1([0,\infty))$ with $\phi'$ satisfying 
\inc{p-1} and \dec{q-1} for some $1<p\leq q$. 
Then, for $\kappa\in(0,\infty)$ and $x,y\in \R^N$, 
%for $\kappa\in(0,\infty)$ and $x,y\in \Rn$ the following hold:
%\begin{enumerate}
%%\item
%%$t\phi'(t)\approx\phi(t)$ and $\phi$ satisfies \inc{p} and \dec{q};
%%\vspace{6pt}
%\item
%$\displaystyle \frac{\phi'(|x|+|y|)}{|x|+|y|}|x-y|^2\approx
%\Big(\frac{\phi'(|x|)}{|x|}x-\frac{\phi'(|y|)}{|y|}y\Big) \cdot (x-y)$;
%\vspace{6pt}
%\item
%$\displaystyle \frac{\phi'(|x|+|y|)}{|x|+|y|}|x-y|^2 
%\lesssim \phi(|x|) -\phi(|y|)-\frac{\phi'(|y|)}{|y|}y\cdot(x-y)$; 
%\vspace{6pt}
%\item
\[
\phi(|x-y|) \lesssim 
\kappa\left[\phi(|x|)+\phi(|y|)\right]+ \kappa^{-1}\frac{\phi'(|x|+|y|)}{|x|+|y|}|x-y|^2.
\]
%\end{enumerate}
\end{proposition} 

%%%%%%%%%%%%%%%%%%%%%%%%%%%%%%%%%%%%%%%%%%%%%%%%%%%%%%%%%%%%%%%%%%%%%%%%%%%%
%\subsection{Generalized Orlicz spaces}

\medskip

Let $L^0(\Omega)$ is the set of the 
measurable functions on $\Omega$. For $\phi\in\Phiw(\Omega)$, the \textit{generalized Orlicz space} (also known as the \textit{Musielak--Orlicz space}) is defined by 
\[
L^{\phi}(\Omega):=\big\{f\in L^0(\Omega):\|f\|_{L^\phi(\Omega)}<\infty\big\},
\] 
with the (Luxemburg) norm 
\[
\|f\|_{L^\phi(\Omega)}:=\inf\bigg\{\lambda >0: \varrho_{\phi}\Big(\frac{f}{\lambda}\Big)\leq 1\bigg\},
\ \ \text{where}\ \ \varrho_{\phi}(f):=\int_\Omega\phi(x,|f(x)|)\,dx.
\]
We denote by $W^{1,\phi}(\Omega)$ the set of $f\in L^{\phi}(\Omega)\cap W^{1,1}(\Omega)$ satisfying that $|Df| \in L^{\phi}(\Omega)$ with the norm $\|f\|_{W^{1,\phi}(\Omega)}:=\|f\|_{L^\phi(\Omega)}+\|\,|Df|\|_{L^\phi(\Omega)}$. Note that if $\phi$ satisfies \adec{q} for some $q\ge 1$, then $f\in L^\phi(\Omega)$ if and only if $\varrho_\phi(f)<\infty$, and if $\phi$ satisfies \azero{}, \ainc{p} and \adec{q} for some $1<p\leq q$, then $L^\phi(\Omega)$ and $W^{1,\phi}(\Omega)$ are reflexive Banach spaces. In addition, we denote by $W^{1,\phi}_0(\Omega)$ the closure of $C^\infty_0(\Omega)$ in $W^{1,\phi}(\Omega)$. For more information about the generalized Orlicz and Orlicz--Sobolev spaces, we refer to the monographs  
\cite{CheGSW21, HarH19, LanM19} and also \cite[Chapter~2]{DieHHR11}.

\smallskip

%%%%%%%%%%%%%%%%%%%%%%%%%%%%%%%%%%%%%%%%%%%%%%%%%%%%%%%%%%%%%%%%%%%%%%%%%%%%
\subsection{New conditions}\label{subsect:newcondition}

The condition \aone{} was introduced in \cite{Has15} (see also \cite{MaeMOS13a}) and is essentially 
optimal for the boundedness of the maximal operator in generalized Orlicz spaces. 
It also implies the H\"older continuity of solutions and (quasi)minimizers 
\cite{BenHHK21, HarHL21, HarHT17}. For higher regularity, 
we introduced in \cite{HasO22} a ``vanishing \aone{}'' condition, denoted \VA{}, and 
a weak vanishing version, \wVA{}. These previous studies applied to 
$\Phi$-functions $\phi:\Omega\times [0,\infty)\to [0,\infty)$ and discussed the sharpness of the conditions. 

We generalize the conditions to the non-Uhlenbeck situation. 
It can be easily seen that \VA{}$\Longrightarrow$\wVA{}$\Longrightarrow$\aone{}. 
The results of this paper require only \wVA{}, but \VA{} is included since it is 
simpler to check and is a sufficient condition. 

\begin{definition}
Let $G:\Omega\times \Rn\to \R^N$, $N\in\mathbb N$, $\epsilon\in[0,1]$, $K, \bL>0$, $r\in (0,1]$ 
and $\omega:[0,1]\to [0,1]$. We consider the inequality 
\[
|G(x,\xi)-G(y,\xi)|\leq \bL\omega(r)\big(|G(y,\xi)|+1\big) 
\quad\text{when }\ |G(y,\xi)|\in [0,K|B_r|^{-1+\epsilon}]
\]
for all $x,y\in B_r\cap \Omega$ and $\xi\in\Rn$.
We say that $G$ satisfies:
\begin{itemize}
\item[\normalfont(A1)]\label{aone}
if for any $K>0$ there exists $\bL=\bL_K>0$ such that the inequality holds for 
$\omega\equiv 1$ and $\epsilon=0$.
\item[\normalfont(VA1)]\label{VA1}
if for any $K>0$ there exists $\bL=\bL_K>0$ and there exists a modulus of continuity 
$\omega:[0,1]\to [0,1]$ such that the inequality holds for 
$\epsilon=0$.
\item[\normalfont(wVA1)]\label{wVA1}
if for any $K>0$ there exists $\bL=\bL_K>0$ and for every $\epsilon>0$ there exists a modulus of continuity $\omega=\omega_\epsilon:[0,1]\to [0,1]$ such that the inequality holds.
\end{itemize}
Here, a modulus of continuity $\omega$ means that $\omega$ is concave and nondecreasing and satisfies $\lim_{r\to 0^+}\omega(r)=\omega(0)=0$.
\end{definition}

If we compare these conditions in the case $N=1$ with the previously mentioned conditions 
in earlier papers, we see that they do not look exactly the same. First of all 
the earlier conditions correspond to the case $K=1$ only. If we consider quasi-convex domains $\Omega$ 
and $\phi$ satisfies \adec{}, 
then the condition with any $K$ is equivalent to the condition with fixed $K$. 
This is proved by a simple chain argument, see \cite[Lemma~3.3]{HarH19b} for \aone{}.  

Furthermore, the earlier conditions did not allow $\xi$ satisfying $|G(y,\xi)|<\omega(r)$. 
For instance, the old formulation of \VA{} (for $G(x,\xi)=\phi(x,|\xi|)$) was to assume that 
\[
|\phi(x,t)-\phi(y,t)|\le L \omega(r) \phi(y,t) 
\quad\text{when } \phi(y,t)\in [\omega(r), |B_r|^{-1}]
\] 
Compared to new \VA{}, the right-hand side has $\phi(y,t)$ instead of $\phi(y,t)+1$ but on the 
other hand the inequality is not assumed for $\phi(y,t)\in [0,\omega(r))$. 
However, we can show that these are again equivalent if $\phi\in \Phiw(\Omega)$ is continuous in $t$.  
A similar argument applies to the other conditions as well. 

Let us show that the old version of \VA{} implies \VA{} with the same $\omega$.
Indeed, we need only show that the old condition implies \VA{} when $\phi(y,t)<\omega(r)$ 
since this is trivial when $\phi(y,t)\in [\omega(r), |B_r|^{-1}]$. 
Fix $B_r$ and let $y\in B_r$. Since $\phi$ is increasing and continuous, we can find $s>t$ with $\phi(y,s)=\omega(r)$. Then by the old condition,
\[
\phi(x,t)\le \phi(x,s)\le (1+L \omega(r)) \phi(y,s) \le (L+1)\omega(r) \quad \text{for all }\ x\in B_r. 
\] 
Hence we obtain \VA{} as follows:
\[
|\phi(x,t)-\phi(y,t)| \le \phi(x,s)+\phi(y,s) \le (L+2)\omega(r)\le (L+2)\omega(r)(\phi(y,t)+1).
\]
For the opposite implication, if $\phi(y,t)\in [\omega(r)^\frac{1}{2}, |B_r|^{-1}]$, then 
\[
\omega(r) (\phi(y,t) + 1) \le 2 \omega(r)^\frac{1}{2}\phi(y,t). 
\]
It follows that \VA{} with $\omega$ implies the old version of \VA{} with 
$\omega^{\frac{1}{2}}$ in place of $\omega$. 

In summary, under minimal and natural assumptions on $G$, our conditions are equivalent to previous versions. 
In particular, from \cite[Chapter~7]{HarH19} and \cite[Section~8]{HasO22}, 
we know that \aone{} is equivalent to logarithmic Hölder continuity in the 
variable exponent case $\phi(x,t)=t^{p(x)}$ while \VA{} and \wVA{} hold with $\omega$ if 
and only if $p\in C^{0,\omega}$. For the double phase case $\phi(x,t)=t^p + a(x)t^q$, 
\VA{} is more or less equivalent to $\frac qp < 1+\frac\alpha n$ while 
\wVA{} and \aone{} require only $\frac qp \le 1+\frac\alpha n$, where $a\in C^{0,\alpha}(\Omega)$. 

\begin{remark}
Baroni, Colombo and Mingione \cite{BarCM18, ColM15-2} initiated the study of double phase 
with additional information on the minimizer. Specifically, they considered bounded 
or H\"older continuous minimizers and showed that the assumption $\frac qp \le 1+\frac\alpha n$ 
can be relaxed in these cases. Ok \cite{Ok20} considered analogous results with Lebesgue 
integrability information on the minimizer. In \cite{BenHHK21, HarHL21}, the generalized 
Orlicz case with additional information was considered, but only for lower regularity, i.e.\ 
Harnack's inequality and H\"older continuity for some $\gamma>0$. 
It is possible, but non-trivial, to develop the ideas from this paper to prove maximal 
regularity in the generalized Orlicz case with additional information; we will return to this issue in a follow-up paper \cite{HasO_pp}.
\end{remark}

%%%%%%%%%%%%%%%%%%%%%%%%%%%%%%%%%%%%%%%%%%%%%%%%%%%%%%%%%%%%%%%%%%%%%%%%%%%%%%%%%%%%
%%%%%%%%%%%%%%%%%%%%%%%%%%%%%%%%%%%%%%%%%%%%%%%%%%%%%%%%%%%%%%%%%%%%%%%%%%%%%%%%%%%%
%%%%%%%%%%%%%%%%%%%%%%%%%%%%%%%%%%%%%%%%%%%%%%%%%%%%%%%%%%%%%%%%%%%%%%%%%%%%%%%%%%%%

\section{Construction of growth function}\label{sect:constructionphi}

In this section, we construct an auxiliary function $\phi\in \Phic(\Omega)$ that measures the growth 
of $A:\Omega\times\Rn\to\Rn$ or $F:\Omega\times \Rn\to[0,\infty)$. 
%In the next definition, we use the same parameter $p$ as for $A$ or $F$, but a different $q$. 

\begin{definition}\label{def:growthfunction}
Let $A:\Omega\times \Rn \to \Rn$ satisfy Assumption~\ref{ass:A}(1).
%and $p$ be as in Assumption~\ref{ass:A}. 
We say that  $\phi\in \Phic(\Omega)$ is a \textit{growth function of $A$} if 
the following conditions are satisfied:
$\phi(x,\cdot)\in C^1([0,\infty))$ for every $x\in\Omega$; 
$\phi'$ satisfies \azero{}, \inc{p_1-1} and \dec{q_1-1} for some $1<p_1\le q_1$; 
and 
\begin{equation}\label{Agrowth}
|A(x,\xi)|+ |\xi|\,|D_\xi A(x,\xi)|\le \Lambda \phi'(x,|\xi|),
\end{equation}
\begin{equation}\label{Aellipticity}
D_\xi A(x,\xi)\tilde \xi \cdot \tilde \xi \ge \nu \frac{\phi'(x,|\xi|)}{|\xi|}|\tilde \xi|^2
%\qquad\text{\color{red}(ellipticity)}
\end{equation} 
for some $0<\nu\le \Lambda$ and for all $x\in \Omega$ and $\xi, \tilde \xi\in \Rn \setminus\{ 0\}$. 

Let $F:\Omega\times \R\to \R$ satisfy Assumption~\ref{ass:F}(1). 
We say that $\phi\in \Phic(\Omega)$ is a \textit{growth function of $F$} 
if it is a growth function of $A=D_\xi F$ in the above sense.  
\end{definition}

\begin{remark}\label{rmk-growth}
In the above definition, we may assume without loss of generality that $\phi(x,\cdot)\in C^2((0,\infty))$. Indeed, define $\psi'(x,t):=\int_0^t\frac{\phi'(x,s)}s\,ds$ and $\psi(x,t):=\int_0^t\psi'(x,s)\,ds$. Then $\psi(x,\cdot)\in C^2((0,\infty))$ satisfies $\phi\approx \psi$, $\min\{1,p_1-1\}\phi'(x,t)\le \psi'(x,t) \le \max\{1,q_1-1\} \phi'(x,t)$, and \inc{p_1-1} and \dec{q_1-1} with the same  $p_1$ and $q_1$.
\end{remark}

\begin{proposition}\label{prop:growthfunction}
Every $A:\Omega\times \Rn \to \Rn$ satisfying Assumption~\ref{ass:A} has 
a growth function $\phi\in \Phic(\Omega)$ with constants $1< p_1\le q_1$ and $0<\nu\le \Lambda$ depending only 
on the parameters $p$, $q$ and $L$  in Assumption~\ref{ass:A}. Specifically, $p_1=p$ and $q_1\ge q$.
\end{proposition}

\begin{proof}
Define $\psi\in \Phiw(\Omega)$ by 
\[
\psi(x,t):=\int_0^t\psi'(x,s)\,ds 
\quad\text{and}\quad 
\psi'(x, t) := \sup_{|\xi|=t} |\xi|\,|D_\xi A(x,\xi)|.
\]
By Assumption~\ref{ass:A}(2), $\psi'$ satisfies \azero{}, \ainc{p-1} and \adec{q-1}, and hence $\psi$ satisfies \azero{}, \ainc{p} and \adec{q} (Proposition~\ref{prop0}(1)). 
In view of \cite[Lemmas 2.2.1 and 2.2.6]{HarH19}, there exists $\widetilde\psi\in \Phic(\Omega)$ with $\widetilde\psi\approx \psi$ which satisfies \azero{}, \inc{p} and \dec{q_1} for some $q_1\ge q$.
Next we define $\phi\in \Phic(\Omega)$ with $\phi(x,\cdot) \in C^1([0,\infty))$ for every $x\in\Omega$ by
\[
\phi(x,t):=\int_0^t \frac{\widetilde\psi (x,s)}{s}\,ds.
\] 
Then $\phi'(x,t)= \frac{\widetilde\psi(x,t)}{t} \approx \frac{\psi(x,t)}{t}\approx \psi'(x,t)$ and 
$\phi'$ satisfies \azero{}, \inc{p-1} and \dec{q_1-1}.

We now show that $\phi$ is the desired  growth function. Note that by Assumption~\ref{ass:A}(2),
\[\begin{split}
|A(x,\xi)| & =\left|\int_0^1\frac{d}{dt} A (x,t\xi) \,dt\right| 
=\left|\int_0^1 D_\xi A(x,t\xi)\cdot \xi \,dt\right| 
\le \int_0^1 |D_\xi A(x,t\xi)|\,|\xi| \,dt \\
& \le L \left(\int_0^1 t^{p-2} \,dt\right) |D_\xi A(x,\xi)|\,|\xi| 
\le\frac{L}{p-1} \psi'(x,|\xi|). %|\partial A(x,\xi)|\,|\xi|.
\end{split} \]
Then \eqref{Agrowth} follows from   $\phi'\approx \psi'$ and the definition of $\psi'$, and \eqref{Aellipticity} holds by  Assumption~\ref{ass:A}(3) with $\phi'\approx \psi'$ and the definition of $\psi'$.
%
%The converse is clear by \eqref{phiAequiv}.
\end{proof}

From now on, if $A:\Omega \times \Rn\to\Rn$ satisfies Assumption~\ref{ass:A}, we take its growth function as the one constructed in the above proposition, hence the constants $p_1$, $q_1$, $\nu$, $\Lambda$ in Definition~\ref{def:growthfunction} depend on the constants $p$, $q$, $L$ in Assumption~\ref{ass:A}.

\begin{remark} Inequality \eqref{Aellipticity} with $\phi'$ satisfying \inc{p-1} and \dec{q_1-1} implies the following strict monotonicity condition: 
\begin{equation}\label{monotonicity}
(A(x,\xi)-A(x,\tilde \xi)) \cdot (\xi-\tilde \xi) \gtrsim \frac{\phi'(x,|\xi|+|\tilde \xi|)}{|\xi|+|\tilde \xi|}|\xi - \tilde \xi|^2, \qquad x\in\Omega, \ \ \xi,\tilde \xi\in \Rn\setminus \{0\}.
\end{equation}
Furthermore, letting $\tilde \xi \to 0$ in the preceding inequality, we also have the coercivity condition:
\[
A(x,\xi) \cdot \xi \gtrsim \phi'(x,|\xi|)|\xi| \approx \phi(x,|\xi|), \qquad x\in\Omega, \ \ \xi\in \Rn
\]
which yields
\begin{equation}\label{phiAequiv}
\phi(x,|\xi|) \lesssim A(x,\xi)\cdot \xi \le |\xi| |A(x,\xi)| 
\lesssim |\xi|\phi'(x,|\xi|) \approx \phi(x,|\xi|), \qquad x\in\Omega,\ \ \xi\in \Rn.
\end{equation} 
\end{remark}

From now on, we denote
\[
\intA(x,\xi ):=|\xi |\,A(x,\xi ),
\]
and write $\phi\approx \intA$ as an abbreviation of \eqref{phiAequiv}. (The notation is chosen since $A^{(-1)}$ 
is comparable to the antiderivative of $A$.) Since \aone{} is invariant under equivalence 
\cite[Lemma~4.1.3]{HarH19} we obtain the following result:

\begin{proposition}\label{prop:Aphiaone}
Let $A:\Omega\times \Rn \to \Rn$ satisfy Assumption~\ref{ass:A} and $\phi$ be its growth function. Then 
$\intA$ satisfies \aone{} if and only if $\phi$ satisfies \aone{}.
\end{proposition}
%
%\begin{proof}
%Suppose that $A$ satisfies \aone. Suppose $\phi_{B_r}^-(t)\le K|B_r|^{-1}$, and 
%put $\xi=(t,0,\dots,0)\in\Rn$. Then there exists $x'\in B_r$ such that 
%$\phi_{B_r}^-(t)\le \phi(x',t)\le 2\phi_{B_r}^-(t)$. By \eqref{phiAequiv} that 
%\[
%|\xi|\,|A(x',\xi)| \le 2Kc|B_r|^{-1}.
%\]
%%Moreover, by \ainc{p-1} of $t\mapsto t |A(x,te)|$, there exists $s=s(p,L)\in(0,1)$ such that $|sz|\,|A(x',sz)|\le |B_r|^{-1}$. 
%Therefore, %by \aone{} and the \ainc{} and \adec{} conditions of $\phi$ and $A$, 
%%\[
%%\phi(x,t)\approx \phi(x,s t) \approx |sz| |A(x,sz)| \overset{\aone}{\approx} |sz| |A(x',sz)| \approx \phi(x',st) \le 2 \phi^-_{B_r}(t), \quad x\in B_r.
%%\]
%\[
%\phi(x,t) \approx |z| |A(x,z)| \overset{\aone}{\approx} 
%|z| |A(x',z)| \approx \phi(x',t) \le 2 \phi^-_{B_r}(t), \quad x\in B_r.
%\]
%This means $\phi$ satisfies \aone. The converse can be proved in an almost same way, hence we omit the proof.
%\end{proof}

We next consider the continuity hypotheses \wVA{}. This assumption is not invariant 
under equivalence, but we have the following implications for the growth function. 
The second part of the proposition is a stronger version of \aone{}, where we notice that the range of $\phi^-_{B_r}(t)$ is changed to $[\omega(r),|B_r|^{-1}]$, which is needed later on. Here $p'$ denotes the H\"older conjugate exponent of $p$, i.e. $p'=\frac p{p-1}$. 

\begin{proposition}\label{prop:AphiVA}
Let $A:\Omega\times \Rn \to \Rn$ satisfy Assumption~\ref{ass:A} and $\phi$ be its growth function. If $\intA$ satisfies \wVA{}, then for each $\epsilon\in(0,1]$ with $\omega=\omega_\epsilon$ and $r\in (0,1]$,
\begin{enumerate}
\item 
for all $x,y\in B_r\cap \Omega$ and 
$\xi\in\R^n$ satisfying $\phi_{B_r}^-(|\xi|)\in [0, |B_r|^{-1+\epsilon}]$,
\[
|A(x,\xi)-A(y,\xi)| \le c\omega(r)^\frac1{p'}\big((\phi')^-_{B_r}(|\xi|)+1\big);
\]
\item 
for all $t>0$ satisfying $\phi^-_{B_r}(t)\in[\omega(r),|B_r|^{-1}]$,
\[
\phi^+_{B_r}(t) \le c \phi^-_{B_r}(t).
\]
\end{enumerate}
Here the constants $c>0$ depend on $p$, $q$, $L$ and $\bL_{2c_1}$, where $c_1\ge1$ is the implicit constant from $\phi\approx \intA$ \eqref{phiAequiv}.  
\end{proposition}
\begin{proof}
Suppose that $\xi\in\Rn$ satisfies $\phi_{B_r}^-(|\xi|)\le |B_r|^{-1+\epsilon}$. Choose $x'\in B_r$ with $\phi(x',|\xi|) \le 2\phi_{B_r}^-(|\xi|)$ so that 
$\phi(x',|\xi|)\le 2|B_r|^{-1+\epsilon}$. 
This together with $\intA\approx \phi$  \eqref{phiAequiv} yields 
\[
|\intA(x',\xi)| \le c_1 \phi(x', |\xi|)\le 2c_1|B_r|^{-1+\epsilon}.
\]
Therefore, by \wVA{} with $K=2c_1$ we have that for any $x,y\in B_r$
\begin{equation}\label{prop:AphiVA-pf1}\begin{split}
&|\intA(x,\xi)-\intA(y,\xi)| \\
&\qquad \le |\intA(x,\xi)-\intA(x',\xi)|+|\intA(x',\xi)-\intA(y,\xi)| \\
&\qquad \le 2\bL_{2c_1} \omega(r)(|\intA(x',\xi)|+1)\approx \omega(r)\left( \phi(x',|\xi|)+1\right) 
\approx \omega(r)( \phi^-_{B_r}(|\xi|)+1)
\end{split}
\end{equation}
and, using $\intA\approx \phi$,
\begin{equation}\label{prop:AphiVA-pf2}
\phi(x,|\xi|) \lesssim |\intA(x,\xi)-\intA(x',\xi)| + \phi(x',|\xi|)  \lesssim \omega(r)( \phi^-_{B_r}(|\xi|)+1)+\phi^-_{B_r}(|\xi|).
\end{equation}

When also $\phi_{B_r}^-(|\xi|)\ge \omega(r)\in [0,1]$, we can continue this by 
\[
\phi(x,|\xi|) 
\lesssim 
\omega(r)( \phi^-_{B_r}(|\xi|)+1)+\phi^-_{B_r}(|\xi|)
\le
2\phi^-_{B_r}(|\xi|)+\omega(r)
\le
3\phi^-_{B_r}(|\xi|).
\]
Since the implicit constant in the preceding inequality is independent of $\epsilon$, we obtain the inequality in (2) for $\xi$ with $\phi_{B_r}^-(|\xi|)\in [\omega(r), |B_r|^{-1})$. Moreover, the continuity of $\phi$ in the $t$ variable implies the same inequality for $\xi$ with $\phi_{B_r}^-(|\xi|)= |B_r|^{-1}$. This completes the proof of claim (2). 

Still considering $\phi_{B_r}^-(|\xi|)\ge \omega(r)\in [0,1]$, we move on to prove (1).  
Since $\phi$ satisfies 
\azero{} and \inc{p}, we conclude that $|\xi|\gtrsim \omega(r)^{1/p}$. 
Dividing \eqref{prop:AphiVA-pf1} by $|\xi|$, we obtain 
\[\begin{split}
|A(x,\xi)-A(y,\xi)| 
\lesssim 
\omega(r)( \phi^-_{B_r}(|\xi|)+1) |\xi|^{-1}
&\le
\omega(r)( (\phi')^-_{B_r}(|\xi|)+\omega(r)^{-1/p})\\
&\le
\omega(r)^{\frac{1}{p'}}( (\phi')^-_{B_r}(|\xi|)+1), 
\end{split}\]
where we used that $t\phi'(x,t)\approx \phi(x,t)$. 

It remains to prove (1) when $\phi_{B_r}^-(|\xi|)< \omega(r)$. By \eqref{prop:AphiVA-pf2}, it follows that $\phi_{B_r}^+(|\xi|)\lesssim \omega(r)$. Hence applying Proposition~\ref{prop0}(4), for any $x\in B_r$,
\[
|A(x,\xi)| \approx \phi'(x,|\xi|) = (\phi^*)^{-1}(x,\phi^*(x,\phi'(x,|\xi|))) \lesssim (\phi^*)^{-1}( x,\omega(r))  \lesssim \omega(r)^{\frac{1}{p'}}
\]
and so (1) follows in this case also. 
\end{proof}

Let $F:\Omega\times \Rn \to [0,\infty)$ satisfy Assumption~\ref{ass:F} and 
$\phi\in\Phic(\Omega)$ be its growth function. By \eqref{phiAequiv} with $A=D_\xi F$,
\begin{equation}\label{phifequiv}
F(x,\xi) = \int_0^1 D_\xi F(x,t\xi)\cdot \xi\,dx \approx \phi(x,|\xi|) \quad \text{for all $x\in\Omega$ and $\xi\in\Rn$,}
\end{equation}
In the same way as Propositions~\ref{prop:Aphiaone} and \ref{prop:AphiVA}, 
we can also prove the following. 

%\begin{proposition}
%Let $F:\Omega\times \Rn \to \R$  be as in Assumption~\ref{ass:F}, and $\phi$ be the growth function of $F$, and $\epsilon\in[0,1)$. Then $F$ satisfies \aone{} if and only if $\phi$ satisfies \aone.
%\end{proposition}

\begin{proposition}\label{prop:fphiVA}
Let $F:\Omega\times \Rn \to \R$ satisfy Assumption~\ref{ass:F} and $\phi$ be its 
growth function. Then $F$ satisfies \aone{} if and only if $\phi$ satisfies \aone{}.

If $F$ satisfies \wVA{}, then for each $\epsilon\in(0,1]$ with $\omega=\omega_\epsilon$ and $r\in(0,1]$,
\begin{enumerate}
\item 
for all $\xi\in\R^n$ satisfying $\phi_{B_r}^-(|\xi|)\in [0, |B_r|^{-1+\epsilon}]$,
\[
F^+_{B_r}(\xi)- F^-_{B_r}(\xi) \le c \omega(r)\big(\phi^-_{B_r}(|\xi|)+1\big)
\]
\item 
for all $t>0$ satisfying $\phi^-_{B_r}(t)\in[\omega(r),|B_r|^{-1}]$,
\[
\phi^+_{B_r}(t) \le c \phi^-_{B_r}(t).
\]
\end{enumerate}
Here the constants  $c>0$ depend on $p$, $q$, $L$ and $\bL_{2c_2}$, where $c_2\ge 1$ is the implicit constant from $\phi \approx F$ \eqref{phifequiv}. 
\end{proposition}

Finally, we consider the relation between \wVA{} of $A$ and $F$. Based on this result we can say that 
the condition \wVA{} is weaker in the minimization case than in the PDE case.

\begin{proposition}\label{prop:hphf}
Let $F:\Omega\times \Rn \to \R$ satisfy Assumption~\ref{ass:F} 
and set $A^{(-1)}(x,\xi):=|\xi| D_\xi F(x,\xi)$. If $\intA$ satisfies \wVA{}, then so does $F$ with the same $\omega$.
\end{proposition}
\begin{proof}
Let $\phi$ be a growth function of $F$. Fix any $B_r\subset \Rn$ with $r\le 1$.
Suppose $F^-_{B_r}(\xi)\in[0, K|B_r|^{-1+\epsilon}]$. Since
$|\intA(y,\xi)|\le c_1\phi(y,|\xi|) \le c_1c_2 F(y,\xi) $
for all $\xi\in \Rn$ by $\phi\approx \intA$ \eqref{phiAequiv} and $\phi\approx F$ \eqref{phifequiv} with implicit constants $c_1$ and $c_2$, respectively, we obtain
\[
|\intA(y,\xi)|\in[0,Kc_1c_2 |B_r|^{-1+\epsilon}].
\] 
Thus we can use \wVA{} of $\intA$. Define $e=\frac{\xi}{|\xi|}$. Then
\[\begin{split}
F(x,\xi)-F(y,\xi) 
&= \int_0^{|\xi|} (A(x,te)-A(y,te))\cdot e\, dt
 \le \bL_{Kc_1c_2}\omega(r)\int_0^{|\xi|} \big[|A(y,te)|+1\big]\, dt \\
& \lesssim \omega(r)\int_{0}^{|\xi|} [\phi'(y,t)+1]\, dt 
= \omega(r)(\phi(y,|\xi|)+|\xi|)\\
&\lesssim \omega(r) (\phi(y,|\xi|)+1)\approx \omega(r) (F(y,\xi)+1)
\end{split} \]
for any $x,\tilde x \in B_r\cap \Omega$, which is \wVA{} of $F$.
\end{proof}

%%%%%%%%%%%%%%%%%%%%%%%%%%%%%%%%%%%%%%%%%%%%%%%%%%%%%%%%%%%%%%%%%%%%%%%%%%%%
%%%%%%%%%%%%%%%%%%%%%%%%%%%%%%%%%%%%%%%%%%%%%%%%%%%%%%%%%%%%%%%%%%%%%%%%%%%%
%%%%%%%%%%%%%%%%%%%%%%%%%%%%%%%%%%%%%%%%%%%%%%%%%%%%%%%%%%%%%%%%%%%%%%%%%%%%

\section{Auxiliary regularity results}\label{sect:auxiliary}

In this section we collect results on two types of approximating problems, namely 
non-autonomous problems with Uhlenbeck structure and autonomous problems without Uhlenbeck structure.

%%%%%%%%%%%%%%%%%%%%%%%%%%%%%%%%%%%%%%%%%%%%%%%%%%%%%%%%%%%%%%%%%%%%%
\subsection{Quasiminimizers}\label{subsect:quasiminimzer}
For $\psi\in\Phiw(\Omega)$, we say that $u\in W^{1,\psi}_{\loc}(\Omega)$ is a local \textit{$\psi$-quasi\-mini\-mizer} if there exists $Q\ge 1$ 
such that 
\[%begin{equation}\label{def:quasiminimizer}
\int_{\supp\,(u-v)} \psi(x,|Du|)\,dx \le Q \int_{\supp\,(u-v)} \psi(x,|Dv|)\,dx 
\]%end{equation}
for all $v\in W^{1,\psi}_{\loc}(\Omega)$ with $\supp\,(u-v)\Subset\Omega$. 
Quasiminimizers of energy functionals with generalized Orlicz growth 
have been studied e.g.\ in \cite{BenHHK21, BenK20, HarHL21, HarHK18}.

Let $A:\Omega\times\Rn\to\Rn$ satisfy Assumption~\ref{ass:A} and $\phi$ be its growth function.
%If $\phi$ is a growth function of
%$A$, then $\phi\approx \psi$ for 
%\begin{equation}\label{phiAweak}
%\psi(x,t):= \sup_{|\xi|=t}|\xi|\,|A(x,\xi)|.
%\end{equation}
%Then the assumption implies that $\psi$ satisfies \azero{}, \ainc{p} and \adec{q}, and $A^{(-1)}\approx \psi$ \eqref{phiAequiv}.  
Let $u\in W^{1,1}_{\loc}(\Omega)$ be a local weak solution to \eqref{mainPDE}. 
By 
$\phi\approx \intA$ \eqref{phiAequiv} and the assumption $|\intA(\cdot,Du)| \in L^1_{\loc}(\Omega)$ we see that $u \in W^{1,\phi}_{\loc}(\Omega)$.
Assuming density, we obtain by approximation that 
\[
\int_\Omega A(x,Du)\cdot D\zeta \,dx =0 \quad \text{for every }\ \zeta\in W^{1,\phi}(\Omega) \ \text{ with }\ \mathrm{supp}\,(\zeta)\Subset \Omega.
\]
Using $\phi\approx \intA$, the identity above 
with $\zeta=u-v$, 
Young's inequality and Proposition~\ref{prop0}(4), we have 
\[
\begin{split}
\int_{\supp\,(u-v)} \phi(x,|Du|)\,dx & \lesssim \int_{\supp\,(u-v)} A(x,Du)\cdot Du\,dx 
= 
\int_{\supp\,(u-v)} A(x,Du)\cdot Dv\,dx \\
& \lesssim 
\int_{\supp\,(u-v)} \frac{\phi(x,|Du|)}{|Du|} |Dv| \,dx \\
& \lesssim 
\epsilon \int_{\supp\,(u-v)} \phi(x,|Du|)\,dx +\epsilon^{1-q}\int_{\supp\,(u-v)} \phi(x,|Dv|)\,dx
\end{split}
\]
for every $\epsilon\in(0,1]$ and $v\in W^{1,\phi}(\Omega)$ with $\supp\,(u-v) \Subset\Omega$. 
Thus
\[
\int_{\supp\,(u-v)} \phi(x,|Du|)\,dx \le Q \int_{\supp\,(u-v)} \phi(x,|Dv|)\,dx 
\]
for some $Q\ge 1$ depending only on $p$, $q$ and $L$ so that $u$ is a $\phi$-quasiminimizer. 
Note that the density of smooth functions holds if $\phi$ satisfies \azero{}, \aone{} and \adec{} 
\cite[Theorem~6.4.7]{HarH19}. 

Thus we can apply the regularity results for quasiminimizers in \cite{HarHL21} to a local weak solution 
to \eqref{mainPDE}. In particular, we have the following H\"older regularity and higher integrability results for \eqref{mainPDE}. For the last estimate, we refer to \cite[Theorem~4.7]{HasO22}.

\begin{theorem}[\cite{HarHL21}]\label{thm:holder} 
Let $A:\Omega\times \Rn\to\Rn$ satisfy Assumption~\ref{ass:A} with constants 
$L\ge 1$ and $1<p \le q$ and have growth function $\phi\in \Phic(\Omega)$ and let $u\in W^{1,1}_\loc(\Omega)$ be a local weak solution 
to \eqref{mainPDE}. If $A^{(-1)}$ satisfies \aone{} with constant $\bL_K>0$,
then
\begin{itemize}
\item $u\in C^{\alpha}_{\loc}(\Omega)$ for some $\alpha=\alpha(n,p,q,L,\bL_1)\in(0,1)$;
\item $A^{(-1)}(\cdot, Du)\approx \phi(\cdot,|Du|)\in L^{1+\sigma}_{\loc}(\Omega)$ for some $\sigma=\sigma(n,p,q,L,\bL_1)>0$. 
Moreover, for some $c=c(n,p,q,L,\bL_1)>0$ we have 
\begin{equation*}%\label{revestimate}
\left(\fint_{B_r} \phi(x,|Du|)^{1+\sigma}\,dx\right)^{\frac{1}{1+\sigma}} 
\le
 c\left(\phi_{B_{2r}}^-\left(\fint_{B_{2r}} |Du|\,dx\right) +1 \right)
\end{equation*}
whenever $B_{2r}\Subset\Omega$ with $|B_{2r}|\le 1$ and $\|Du\|_{L^\phi(B_{2r})}\le1$.
\end{itemize}
\end{theorem}

\begin{remark}\label{rmk:holder_functional}
The above theorem still holds in the setting of the minimization problem as in Theorem~\ref{thm:functional}, where we assume that $F$ satisfies \aone{}. 
\end{remark}

\begin{remark}\label{rmk:phipsi}
Assumptions~\ref{ass:A} and \ref{ass:F} involve the derivative of $A$ or the second derivative of 
$F$. In fact, these are only needed for the maximal regularity. For the previous theorem, 
Assumptions~\ref{ass:A} could be replaced by weaker assumptions, such as 
\begin{enumerate}
\item
For every $x\in \Omega$, $A(x, 0)= 0$, $A(x,\cdot)\in C^{0}(\R^n\setminus\{ 0\}; \R^n)$
and for every $\xi\in \R^n$, $A(\cdot,\xi )$ is measurable.
%
%$A(x, 0)= 0$ for all $x\in\Omega$, $A(\cdot,\xi )$ is measurable and 
%$A(x,\cdot)\in C^{0}(\R^n\setminus\{ 0\}; \Rn)$.
\item 
There exist $L\ge 1$ and $1<p<q$ such that the radial function 
$t\mapsto |A(x,te)|$ satisfies \azero{}, \ainc{p-1} and \adec{q-1}
for every $x\in \Omega$ and $e\in \Rn$ with $|e|=1$. %(\textit{Weak $(p,q)$-growth}) 
\item 
There exists $L\ge 1$ such that 
\[
|\xi'|\,| A(x,\xi ')|
\le
L\,  A(x,\xi ) \cdot \xi 
\]
for all  $x\in\Omega$ and $\xi,\xi' \in \Rn$ with $|\xi|=|\xi'|$. %(\textit{Quasi-isotropic coercivity})
\end{enumerate}
In this case the growth function $\phi$ is replaced by
\[
\psi(x,t):= \sup_{|\xi|=t}|\xi|\,|A(x,\xi)|.
\]
\end{remark}

%%%%%%%%%%%%%%%%%%%%%%%%%%%%%%%%%%%%%%%%%%%%%%%%%%%%%%%%%%%%%%%%%%%%%
\subsection{Regularity results for autonomous problems}
We consider Assumption~\ref{ass:A} for an autonomous function $\bA:\Rn\to \Rn$ 
via its trivial extension $\bA(x,\xi):= \bA(\xi)$. We may apply Proposition~\ref{prop:growthfunction}  to conclude that such $\bA$ has a growth function $\bphi\in \Phic(\Omega)$. 
An inspection of the proof of Proposition~\ref{prop:growthfunction} shows that this $\bphi$ can be chosen autonomous as well, 
and we will assume hereafter that $\bphi\in \Phic$, i.e., $\bphi= \bphi(t)$. 
Moreover, in view of Remark~\ref{rmk-growth}, we also  assume  that $\bphi\in C^1([0,\infty))\cap C^2((0,\infty))$.

%In this subsection, we consider 
With these $\bA$ and $\bphi$, we present regularity results of 
weak solutions to the following autonomous problem
\begin{equation}\label{eqA0}
\tag{{\(\div \bA\)}}
\mathrm{div}\, \bA (D\bu)=0 \quad \text{in }\ B_r,
\end{equation}
The following results are variations of known regularity results for equations with Orlicz growth, and the proofs are similar to previous ones. In particular, in the Uhlenbeck case $\bA(\xi)=\frac{\bphi'(|\xi|)}{\xi}\xi$, we considered them in \cite{HasO22}. Therefore, we outline the proofs and point out 
differences compared with the references.
The first result is local $C^{1,\alpha}$-regularity.

%The corresponding result for the Uhlenbeck case $\bA(\xi)=\bA(|\xi|)$ was considered in \cite[Lemma~4.12]{HasO22}.

\begin{lemma}\label{lem:holder}
Let $\bA:\R^n\to\R^n$ satisfy Assumption~\ref{ass:A} with $A(x,\xi)\equiv \bA(\xi)$ and constants $L\ge 1$ and $1<p\le q$, and $\bphi\in \Phic$ be its growth function. 
If $\bu{}\in W^{1,\bphi}(B_r)$ is a weak solution to \eqref{eqA0}, then 
$D\bu\in C^{0,\bar\alpha}_{\loc}(B_r,\Rn)$ for some $\bar\alpha\in (0,1)$ with the following estimates: for any $B_\rho\subset B_r$ and any $\tau\in(0,1)$,
\begin{equation*}%\label{Lip}
\sup_{B_{\rho/2}} |D\bu|\leq c\fint_{B_\rho} |D\bu|\, dx,
\end{equation*} 
%for any $B_\rho\subset B_r$, and
\begin{equation*}%\label{C1alpha0}
\fint_{B_{\tau \rho}}\Big|D\bu(x)-\fint_{B_{\tau \rho}} D\bu(y)\, dy \Big|\,dx 
\leq 
c \tau^{\bar\alpha}\fint_{B_\rho}|D\bu|\,dx.
\end{equation*}
%for any $\tau\in(0,1)$. 
Here $\bar\alpha\in(0,1)$ and $c>0$ depend only on the constants $n$, $p$, $q$ and $L$.
\end{lemma}

\begin{proof}[Outline of proof]
The proof is almost the same as for the Uhlenbeck case, that is \cite[Lemma~4.12]{HasO22} 
with a proof in \cite[Appendix A]{HasO22}. We also refer to \cite{Le1} for the case $\bphi(t)=t^p$. 
We only need to slightly modify the beginning and approximation parts to the general case $\bA(\xi)$.

Instead of the equation \eqref{eqA0}, we consider approximate non-degenerate equations. For $0<\epsilon<\frac{1}{2}$, define
\[
\bA_\epsilon(\xi):= \bA\big( (\tfrac{\epsilon}{|\xi|}+1)\xi\big)\frac{|\xi|}{\epsilon+|\xi|} 
\quad \text{and}\quad 
\bphi_\epsilon' (t):= \frac{\bphi'(\epsilon+t)}{\epsilon+t}t.
\] 
Note that $\bphi_\epsilon(t) := \int_0^t\bphi'_\epsilon(s)\,ds$ is exactly the same as $\phi_\epsilon(t)$ in \cite[(A.1)]{HasO22}, hence, in view of \cite[Appendix A]{HasO22}, $\bphi_\epsilon$ also satisfies \azero{}, \inc{p-1} and \dec{q_1-1}, $\bphi(t) \lesssim \bphi_\epsilon(t)+1$ and $\bphi_\epsilon(t) \lesssim \bphi(t)+1$. 
In particular, the last two inequalities imply $W^{1,\bphi}(B_r)=W^{1,\bphi_\epsilon}(B_r)$.
Then we can see from \eqref{Agrowth} and \eqref{Aellipticity} with $A=\bA$ and $\phi=\bphi$ that  there exist small $\epsilon_0>0$ and $0< \bar \nu \le \bar \Lambda$ depending only on $n$, $p$, $q$, $\nu$ and $\Lambda$ such that for every $\epsilon\in[0,\epsilon_0)$
 \begin{equation}\label{Aepsilon-condition}
 |\bA_\epsilon(\xi)|+ |\xi|\,|D_\xi \bA_\epsilon(\xi)|\le \bar\Lambda \bphi_\epsilon'(|\xi|)
 \quad\text{and}\quad 
D_\xi \bA_\epsilon(\xi)\tilde \xi \cdot \tilde \xi \ge \bar\nu \frac{\bphi_\epsilon'(|\xi|)}{|\xi|}|\tilde \xi|^2
 \end{equation}
for all $\xi,\tilde\xi\in\Rn$.
Moreover, by \eqref{Agrowth},  
\begin{equation}\label{Aepsilon-convergence}
\bA_\epsilon(Dg) \ \ \longrightarrow\ \ \bA(Dg) \quad \text{in }\ L^{\bphi^*}(B_r) \ \ \text{as}\ \ \epsilon\to 0
\end{equation}
for every $g\in W^{1,\bphi}(B_r)$. 

Let $\bu_\epsilon\in W^{1,\bphi}(B_r)$ be the unique weak solution to
\[
\div \bA_\epsilon(D\bu_\epsilon) =0 \quad \text{in }\ B_r 
\quad \text{and} \quad \bu_\epsilon =\bu \quad \text{on }\ \partial B_r.
\]
Fix any small $0<\epsilon \ll 1$. Then, in the same way as in \cite{DieSV17}, in particular Lemmas 5.7 and 5.8, $u_\epsilon \in W^{2,2}_{\loc}(B_r)$, $\bphi_\epsilon(Du_\epsilon)\in W^{1,2}_\loc(B_r)$ and $u_{\epsilon}\in W^{1,\infty}_{\loc}(B_r)$ with the estimate 
\[
\sup_{B_\rho} \bphi_\epsilon (D\bu_\epsilon) \le c \fint_{B_{2\rho}} \bphi_{\epsilon}(D\bu_{\epsilon})\,dx, \quad B_{2\rho}\subset B_r,
\] 
for some $c>0$ depending only on $n$, $p$, $q_1$, $\bar\nu$ and $\bar\Lambda$. 
Set %$a_{ij}=a_{ij}(D\bu_\epsilon)$ and $b_{ij}=b_{ij}(D\bu_\epsilon)$, where
\[ 
[a_{ij}(\xi)]
:= D_\xi A_\epsilon(\xi) 
\quad \text{and} \quad
[b_{ij}(\xi)]
:= D_\xi A_\epsilon(\xi)\frac{\epsilon+|\xi|}{\phi'(\epsilon+|\xi|)}.
%=a_{ij}(\xi)\frac{\epsilon+|\xi|}{\phi'(\epsilon+|\xi|)} .
\]
Note that, by \eqref{Aepsilon-condition}, $b_{ij}(\xi)$ satisfies
\[
\bar \nu  |\tilde \xi|^2 \le \sum_{i,j=1}^{n}b_{ij}(\xi) \tilde \xi_i \tilde \xi_j
\quad\text{and} \quad 
\big|[b_{ij}(\xi)]\big| \le \bar \Lambda.
\]
Therefore, by following \cite[Appendix A]{HasO22}, in particular from the paragraph containing (A.7), 
with the above setting, we can obtain the desired regularity estimates for $u_\epsilon$, 
where relevant  constants are independent of $\epsilon$.

Then using \eqref{Aepsilon-convergence} and the uniform monotonicity \eqref{monotonicity} with $A=\bA_{\epsilon}$ and $\phi=\bphi_\epsilon$, we see that
\[
\bu_\epsilon \ \ \longrightarrow\ \ \bu \quad \text{in}\ \  W^{1,\bphi}(B_r)  \ \ \text{as}\ \ \epsilon\to 0.
\]
Note that the standard Minty-Browder technique for a monotonicity operator, that is, an operator $A$ which satisfies inequality \eqref{monotonicity} with right hand side $0$, implies only weak convergence, but due to the uniform monotonicity we have the strong convergence. In particular, $D\bu_\epsilon(x)$ converges, up to a subsequence, to $D\bu(x)$ almost everywhere in $B_r$. 
\end{proof}

The next result is a Calder\'on--Zygmund type estimate in the generalized Orlicz space for non-zero boundary data. 
Since $\theta$ is superlinear, this lemma allows us to transfer regularity from $u$ to
$\bu$. 
%It generalizes \cite[Lemma~4.15]{HasO22}. 

\begin{lemma}\label{lem:CZ}
 Let $\bA:\R^n\to\R^n$ satisfy Assumption~\ref{ass:A} with $A(x,\xi)\equiv \bA(\xi)$ and constants $L\ge 1$ and $1<p\le q$, and $\bphi\in \Phic$ be its growth function.  Suppose $\theta\in \Phiw(B_r)$ satisfies
\azero{}, \ainc{p_\theta} and \adec{q_\theta} with constants $L_\theta\ge 1$ and $1<p_\theta\leq q_\theta$ and \aone{} with constant $\bL_K>0$, 
and   $u\in W^{1,\bphi}(B_r)$ satisfies $\int_{B_r}\theta(x,\bphi(|Du|))\,dx\le \kappa$ for some $\kappa>0$.
If $\bu\in u+W^{1,\bphi}_0(B_r)$ is a weak solution to \eqref{eqA0}, then 
\[
\|\bphi(|D\bu|)\|_{L^\theta(B_r)} \leq c\, \|\bphi(|Du|)\|_{L^\theta(B_r)}
\]
and
\begin{equation*}%\label{meanCZestimate}
\fint_{B_r}\theta(x,\bphi(|D\bu|))\,dx 
\leq 
c\big(\kappa^{\frac{q_1}{p_1}-1}+1\big)
\bigg(\fint_{B_r}\theta(x,\bphi(|Du|))\,dx + 1\bigg)
\end{equation*}
for $c=c(n,p,q,L, p_\theta,q_\theta,L_\theta,\bL_1)>0$.
\end{lemma}

\begin{proof}[Outline of proof]
The lemma is a general version of \cite[Lemma~4.15]{HasO22} and the proof is exactly the same if we have the general version of \cite[Theorem B.1]{HasO22}. We note that the main tools used in the proof  are weighted $L^p$ estimates, extrapolation in the generalized Orlicz spaces and scaling.  In particular, \cite[Theorem B.1]{HasO22} is the global weighted $L^p$ estimates for the gradient of weak solution to equation \eqref{eqA0} with $\bA(\xi)=\bphi'(|\xi|)\frac{\xi}{|\xi|}$. 

The general version of  \cite[Theorem B.1]{HasO22} can be proved in a similar way. The major difference is establishing the estimate (see the next paragraph) concerned with the boundary Lipschitz regularity for the following approximate equation:
\[
\div \bA(D\bv)=0 \quad \text{in }\ B_{\rho}^+, \qquad \bv=0 \quad \text{on }\ B_{\rho}(0)\cap \{x_n=0\},
\] 
where $B_{\rho}^+=B_{\rho}(0)\cap\{x_n>0\}$. Other differences are minor modifications. 

For a weak solution $\bv\in W^{1,\bphi}(B_\rho^+)$ to the above equation, one can prove that
\[
\sup_{B^+_{\rho/2}}\bphi(|D\bv|) \le c \fint_{B_\rho^+} \bphi(|D\bv|) \, dx
\]
for some $c>0$ depending only on $n$, $p$, $q_1$, $\nu$ and $\Lambda$. This corresponds to  \cite[(B.13)]{HasO22} in the Uhlenbeck case, which was obtained from the interior counterpart by applying the reflection argument with the odd extension, see the last paragraph in \cite{HasO22}. However, in the general case the reflection argument does not  work, and we use a so-called barrier argument. For the %above 
proof of the above Lipschitz estimate we refer to 
 \cite[Theorem 4.1]{Cho18}, see also \cite[Theorem 2.2]{ColM16}. Note that Cho \cite{Cho18} proved a Calder\'on-Zymund type estimate with $\theta(x,t)\equiv \theta(t)$ for nonhomogeneous equations with the zero boundary condition.        
\end{proof}

%%%%%%%%%%%%%%%%%%%%%%%%%%%%%%%%%%%%%%%%%%%%%%%%%%%%%%%%%%%%%%%%%%%%%%%%
%%%%%%%%%%%%%%%%%%%%%%%%%%%%%%%%%%%%%%%%%%%%%%%%%%%%%%%%%%%%%%%%%%%%%%%%
%%%%%%%%%%%%%%%%%%%%%%%%%%%%%%%%%%%%%%%%%%%%%%%%%%%%%%%%%%%%%%%%%%%%%%%%

\section{Approximation}\label{sect:approx}

In this section we construct approximations for $A$ in \eqref{mainPDE} and $F$ in \eqref{mainfunctional}. We start by recalling the simpler construction for $\phi\in C^{1}([0,\infty))$ with $\phi'$ satisfying    \azero{}, \inc{p-1} and \dec{q_1-1} for some $1<p\le q_1$. Let $B_r=B_r(x_0)\subset\Omega$ and  
$0<t_1<1<t_2$. We define 
\begin{equation}\label{phi0}
\bphi(t):=\int_0^t \bphi'(s)\,ds \quad \text{with}\quad \bphi'(t):=
\begin{cases}
\frac {a_1}{t_1^{p-1}}t^{p-1}&\text{if} \ \ 0\leq t\leq t_1,\\
\phi'(x_0,t)&\text{if}\ \ t_1\leq t\leq t_2,\\
\frac {a_2}{t_2^{p-1}}t^{p-1}&\text{if} \ \ t_2\leq t<\infty,
\end{cases}
\end{equation}
where $a_1:=\phi'(x_0, t_1)$ and $a_2 := \phi'(x_0, t_2)$. 
Note that the relationship between $\phi$ and $\bphi$ is exactly the 
same as in \cite[Section~5]{HasO22}. From there, we have the 
following result.

\begin{proposition}\label{prop:phi0}
Let $\phi$ and $\bphi$ be from \eqref{phi0} and $\tL\ge 1$. Suppose that
\begin{equation*}%\label{eq:aones}
\phi_{B_r}^+(t)\le \tL \phi_{B_r}^-(t) \quad \text{for all }\ t\in [t_1,t_2].
\end{equation*}
Then,
\begin{enumerate}
\item 
$\bphi\in C^1([0,\infty))$ with $\bphi'$ satisfying \inc{p-1} and \dec{q_1-1}.
\item
$\bphi(t)\le \phi(x_0,t)$ and $\tL^{-1} \phi(x,t) \le \bphi(t)$ % \le \frac{q_1}{p}\tL \phi(x,t)$ 
for all $(x,t)\in B_r\times [t_1,t_2]$.
\item
$\bphi(t)\leq \frac{q_1}{p}\tilde L \phi(x,t)$ for all $(x,t)\in B_r\times [t_1,\infty)$.
% and so  $\bphi(t) \le \frac{q_1}{p}\tL \left(\phi(x,t) +1\right)$ for all $(x,t)\in B_r\times[0,\infty)$. 
\item 
$\theta_0(x,t):=\phi(x,\bphi^{-1}(t))$ satisfies \azero, \ainc{1} and \adec{q_1/p} 
with constants $L$ depending only on the structure constants of $\phi$.
\end{enumerate}
\end{proposition}

Note that the inequality assumed in the above proposition will be linked to the one in Proposition~\ref{prop:AphiVA}(2) and \ref{prop:fphiVA}(2).
In the Uhlenbeck case $F(x,\xi)=\phi(x,|\xi|)$ a suitably mollified version of $\bphi$ can be used as the approximation of $\phi$ \cite{HasO22}. 
However, as far as we can tell, a similar approach to \eqref{phi0} does 
not work in the general case, so we introduce smooth transitions around $t_1$ and $t_2$ with 
transition functions $\eta_i$. 
We deal with $A$ and $F$ in separate subsections. 

%%%%%%%%%%%%%%%%%%%%%%%%%%%%%%%%%%%%%%%%%%%%%%%%%%%%%%%%%%%%%%%%%%%%
\subsection{Approximation for weak solutions}\label{sec:PDEapprox}

Let $A:\Omega\times \Rn \to \Rn$ satisfy Assumption~\ref{ass:A}, and $\phi\in \Phic(\Omega)$ be 
its growth function. Recall that $\phi'$ satisfies \azero, \inc{p-1} and \dec{q_1-1}. 
In this subsection, we construct an autonomous nonlinearity $\bA :\Rn\to \Rn$ with 
growth function $\bphi$ given by \eqref{phi0}, such that $\bA$ and $\bphi$ are comparable 
with $A$ and $\phi$, respectively, in a suitable sense.

Fix any small ball $B_r=B_r(x_0)\Subset \Omega$ satisfying $|B_r|\le 1$. 
For $t_1\in (0,\frac{1}{2}]$ and $t_2 \ge 2$ that will be chosen later in the next section, 
$a_1$ and $a_2$ from the definition of $\bphi$, 
and constants $q_1$, $\nu$, $\Lambda$ given in Proposition~\ref{prop:growthfunction} we define 
\begin{equation}\label{A0def}
\bA(\xi): = 
\underbrace{\frac{\nu}{8} \eta_1(|\xi|) \frac{a_1}{t_1^{p-1}} |\xi|^{p-2}\xi}_{=:\bA_1(\xi)} 
+ 
\underbrace{\eta_2(|\xi|) A(x_0,\xi)}_{=:\bA_2(\xi)} 
+ 
\underbrace{\bar\Lambda \eta_3(|\xi|) \frac{a_2}{t_2^{p-1}} |\xi|^{p-2}\xi}_{=:\bA_3(\xi)},
\end{equation}
where $\bar\Lambda :=\frac{2^{q_1 -p+3}\Lambda}{\min\{p-1,1\}}$,
and $\eta_i\in C^\infty([0,\infty))$, $i=1,2,3$, satisfy 
\[
\eta_1\equiv 1\ \text{ in } \ [0,t_1), \quad \eta_1\equiv 0\ \text{ in } \ [2t_1,\infty) \quad \text{and}\quad - \tfrac{2}{t_1}\le \eta_1' \le 0,
\]
\[
\eta_2\equiv 1\ \text{ in } \ [0,t_2), \quad \eta_2\equiv 0\ \text{ in } \ [2t_2,\infty) \quad \text{and}\quad -\tfrac{2}{t_2} \le \eta_2'\le 0,
\]
\[
\eta_3\equiv 0\ \text{ in } \ [0,\tfrac{t_2}{2}), \quad \eta_3\equiv 1\ \text{ in } \ [t_2,\infty) \quad \text{and}\quad 0\le \eta_3'\le \tfrac{4}{t_2}.
\]
Clearly, $\bA\in C(\Rn,\Rn)\cap C^{1}(\Rn\setminus\{ 0\},\Rn)$, and 
\begin{equation*}%\label{A0A}
\bA(\xi)=A(x_0,\xi) \quad \text{whenever }\ 2t_1\le |\xi| \le \tfrac{t_2}{2}. 
\end{equation*}
Possible functions $\eta_i$ are sketched in Figure~\ref{fig:functions} to assist in following the proof.

\begin{figure}[ht!]
\definecolor{cqcqcq}{rgb}{0.7529411764705882,0.7529411764705882,0.7529411764705882}
\begin{tikzpicture}[line cap=round,line join=round,>=triangle 45,x=1.0cm,y=2.0cm]
\draw [color=cqcqcq,, xstep=1.0cm,ystep=1.0cm] (-0.1,-0.1) grid (11.1,1.3);
\draw[->,color=black] (-0.1,0.) -- (11.1,0.);
%\foreach \x in {0,1,2,3,4,5,6,7,8,9,10,11}
%\draw[shift={(\x,0)},color=black] (0pt,2pt) -- (0pt,-2pt) node[below] {\footnotesize $\x$};
%\draw[->,color=black] (0.,-0.1) -- (0.,1.3);
\foreach \y in {0,1}
\draw[shift={(0,\y)},color=black] (2pt,0pt) -- (-2pt,0pt) node[left] {\footnotesize $\y$};
\clip(-0.1,-0.4) rectangle (11.1,1.3);
\def\t{0.8};
\def\tt{5};
\def\eps{0.01};
\def\delta{0.1};
\draw[color=black] (\t, -0.2) node {$t_1$};
\draw[color=black] (2*\t, -0.2) node {$2t_1$};
\draw[color=black] (\tt/2+0.1, -0.2) node {$\frac12t_2$};
\draw[color=black] (\tt, -0.2) node {$t_2$};
\draw[color=black] (2*\tt, -0.2) node {$2t_2$};
\draw[color=red, line width=2pt, loosely dashed] (0,1+\eps) -- (\t+\delta,1+\eps) -- (2*\t-\delta,\eps) -- (11,\eps);
\draw[color=red, line width=2pt, loosely dashed] (1.5*\t-0.2, 0.5) node {$\eta_1$};
\draw[color=ForestGreen, line width=1.5pt] (0,1) -- (\tt+\delta,1) -- (2*\tt-\delta,0) -- (11,0);
\draw[color=ForestGreen, line width=1.5pt] (1.5*\tt-0.45, 0.5) node {$\eta_2$};
\draw[color=blue, line width=1.5pt, dotted] (0,-\eps) -- (\tt/2+\delta,-\eps) -- (\tt-\delta,1-\eps) -- (11,1-\eps);
\draw[color=blue, line width=1.5pt, dotted] (0.75*\tt-0.3, 0.5) node {$\eta_3$};
\begin{scriptsize}
%\draw[color=black] (-2.64,-2.81) node {$f$};
\end{scriptsize}
\end{tikzpicture}
\caption{The functions $\eta_i$}\label{fig:functions}
\end{figure}

\begin{lemma}\label{lem:regular}
If $\phi$ is a growth function of $A$, and $\bphi$ and $\bA$ are their autonomous 
approximations as in 
\eqref{phi0} and \eqref{A0def}, then $\bphi$ is a growth function of $\bA$. 
\end{lemma}
\begin{proof}
Let us start by calculating the derivative of $\eta_i(|\xi|) |\xi|^{p-2}\xi$ with $i=1,3$:
\begin{equation}\label{eq:p-hessian}
D_\xi(\eta_i(|\xi|) |\xi|^{p-2}\xi) 
= 
\eta_i'(|\xi|) |\xi|^{p-3}\xi\otimes \xi
+
(p-2)\eta_i(|\xi|) |\xi|^{p-4}\xi\otimes \xi
+
\eta_i(|\xi|) |\xi|^{p-2}I_n.
\end{equation}
Here $\otimes$ denotes the tensor product $(a_i)_i\otimes(b_i)_j = (a_ib_j)_{i,j}$ and 
$I_n$ is the $n$-dimensional identity matrix. 

We consider the first condition \eqref{Agrowth} of being a growth function 
with $\bA$  and $\bphi$ in place of $A$  and $\phi$.
When $|\xi|<2t_1$, we use 
$|\eta_1'(t)|\,t \le 2\frac{t}{t_1}\le 4$ and  $\bphi'(t)\approx \frac{a_1}{t^{p-1}_1}t^{p-1}$ by \inc{p-1} and \dec{q_1-1} for $t_1\le t <2t_1$ to conclude that 
\[
|\bA_1(\xi)|
+ 
|\xi| |D_\xi \bA_1(\xi)|
\le 
\frac\nu8 \frac{a_1}{t_1^{p-1}} 
\big[|\xi|^{p-1}+|\eta_1'(|\xi|)| |\xi|^{p}+(p-1) |\xi|^{p-1}\big]
\lesssim \frac{a_1}{t^{p-1}_1}|\xi|^{p-1} \approx \bphi'(|\xi|).
\]
Similarly, we conclude that $|\bA_3(\xi)|+ |\xi| |D_\xi \bA_3(\xi)|\lesssim 
\frac{a_2}{t_2^{p-1}}|\xi|^{p-1}\approx \bphi'(|\xi|)$ when $|\xi|\ge \frac{1}{2}t_2$. 
Finally, it follows by 
 \eqref{Agrowth} of $A$ and $\phi'(x_0,t)\approx \phi'(x_0,t_2)$ when $t_1\le t \le 2t_2$ that 
$|\bA_2(\xi)|+ |\xi| |D_\xi \bA_2(\xi)|\lesssim \bphi'(|\xi|)$. 
Therefore, we have established \eqref{Agrowth} of $\bA$.

For the ellipticity condition \eqref{Aellipticity}, we consider four cases: $0<|\xi|\le t_1$, $t_1<|\xi|\le 2 t_1$, 
$t_2/2<|\xi|\le t_2$ and $t_2<|\xi|\le 2 t_2$. 
The strategy is to get the ellipticity condition by the same condition of $\bA_1$ and $\bA_2$ for first and third cases, respectively. 
In the second and fourth cases one 
term of $D_\xi \bA$ is non-negative and the other is non-positive (since 
$\eta_1'$ or $\eta_2'$ is non-positive), so we have to show that the non-positive 
term can be absorbed in the non-negative one based on more precise estimates. 

Note that the other cases, $2t_1<|\xi|\le t_2/2$ and $|\xi|>2t_2$, are clear, since 
in these intervals only $\bA_2$ and $\bA_3$, respectively, influence $\bA$.

\framebox{$0<|\xi|\le t_1$} In this interval, $\eta_1\equiv\eta_2\equiv 1$ and 
$\eta_3\equiv 0$. 
By the calculation \eqref{eq:p-hessian} and the ellipticity of $A(x_0,\xi)$,  
\[\begin{split}
D_\xi \bA(\xi) \tilde \xi\cdot \tilde \xi 
& = D_\xi \bA_1(\xi) \tilde \xi\cdot \tilde \xi + 
\underbrace{D_\xi A(x_0, \xi) \tilde \xi\cdot \tilde \xi}_{\ge 0}
\ge 
\frac{\nu}{8} \frac{a_1}{t_1^{p-1}}
[(p-2) |\xi|^{p-4}(\xi\cdot \tilde \xi)^2 + |\xi|^{p-2}|\tilde \xi|^2] \\
& \ge \frac{\nu}{8} \min\{p-1,1\} \frac{a_1}{t_1^{p-1}} |\xi|^{p-2}|\tilde \xi|^2 = \frac{\nu}{8} \min\{p-1,1\} \frac{\bphi'(|\xi|)}{|\xi|} |\tilde \xi|^2.
\end{split}\]

\framebox{$t_1<|\xi|\le 2 t_1$} 
In this interval, $-\frac{4}{|\xi|} \le -\frac{2}{t_1} \le \eta_1' \le 0$, $\eta_2\equiv 1$ and $\eta_3\equiv 0$.
Therefore by the ellipticity of $|\xi|^{p-2}\xi$ and $A(x_0,\xi)$ and \inc{p-1} of $\phi'$,
\[\begin{split}
D_\xi \bA(\xi) \tilde \xi\cdot \tilde \xi 
& = \frac{\nu}{8} \eta'(|\xi|) \frac{a_1}{t_1^{p-1}} |\xi|^{p-3} (\xi\cdot \tilde \xi)^2 
+\underbrace{\frac{\nu}{8} \eta_1(|\xi|)\frac{a_1}{t_1^{p-1}}D_\xi (|\xi|^{p-2}\xi)\tilde \xi \cdot \tilde \xi}_{\ge 0}+ D_\xi A(x_0,\xi) \tilde \xi\cdot \tilde \xi \\
& \ge - \frac{\nu}{8} \frac{4}{|\xi|} \frac{a_1}{t_1^{p-1}} |\xi|^{p-1} |\tilde \xi|^2 + \nu \frac{\phi'(x_0,|\xi|)}{|\xi|}|\tilde \xi|^2 \\
& \ge - \frac{\nu}{2} \frac{\phi'(x_0,|\xi|)}{|\xi|}|\tilde \xi|^2 + \nu \frac{\phi'(x_0,|\xi|)}{|\xi|}|\tilde \xi|^2 = \frac{\nu}{2} \frac{\bphi'(|\xi|)}{|\xi|} |\tilde \xi|^2.
\end{split}\]

\framebox{$\frac{t_2}{2}<|\xi|\le t_2$}
In this interval, $ \eta'_3,\eta_3\ge 0$, $\eta_2\equiv 1$  and 
$\eta_1\equiv 0$. Hence 
\[\begin{split}
D_\xi \bA(\xi) \tilde \xi \cdot \tilde \xi 
& = D_\xi A(x_0,\xi) \tilde \xi \cdot \tilde \xi 
+\underbrace{\bar\Lambda \eta_3'(|\xi|) \frac{a_2}{t_2^{p-1}} |\xi|^{p-3} 
(\xi\cdot \tilde \xi)^2}_{\ge0} + 
\underbrace{\bar\Lambda \eta_3(|\xi|) \frac{a_2}{t_2^{p-1}} \partial (|\xi|^{p-2}\xi) \tilde \xi \cdot \tilde \xi}_{\ge0}\\
& \ge D_\xi A(x_0,\xi) \tilde \xi \cdot \tilde \xi \ge \nu \frac{\phi'(x_0,|\xi|)}{|\xi|} |\tilde \xi|^2 =\nu \frac{\bphi'(|\xi|)}{|\xi|} |\tilde \xi|^2.
\end{split}\]

\framebox{$t_2<|\xi|\le 2 t_2$}
In this interval, $-\frac{4}{|\xi|} \le -\frac{2}{t_2} \le \eta_2'\le 0$, 
$\eta_3\equiv1$ and $\eta_1\equiv 0$. 
Therefore
\[\begin{split}
D_\xi \bA(\xi) \tilde \xi \cdot \tilde \xi
& = 
\frac{\eta_2'(|\xi|)}{|\xi|} (\xi\otimes A(x_0,\xi))\tilde \xi \cdot \tilde \xi + \underbrace{\eta_2(|\xi|)D_\xi A(x_0,\xi) \tilde \xi \cdot \tilde \xi}_{\ge 0} + \bar\Lambda \frac{a_2}{t_2^{p-1}} D_\xi (|\xi|^{p-2}\xi) \tilde \xi \cdot \tilde \xi \\
& \ge 
- 4\Lambda \frac{\phi'(x_0,|\xi|)}{|\xi|} |\tilde \xi|^2 + \bar\Lambda \min\{p-1,1\} \frac{a_2}{t_2^{p-1}} |\xi|^{p-2}|\tilde \xi|^2.
\end{split}\]
For the first term, we use \inc{p-1} and \dec{q_1-1} of $\phi'$ to conclude that 
\[
\frac{\phi'(x_0,|\xi|)}{|\xi|^{p-1}} 
\le
\frac{\phi'(x_0,2t_2)}{(2t_2)^{p-1}} 
\le 2^{q_1-p}\frac{\phi'(x_0,t_2)}{t_2^{p-1}}
=
2^{q_1-p}\frac{a_2}{t_2^{p-1}}.
\]
Recalling the definition of $\bar\Lambda$, we complete the proof of 
\eqref{Aellipticity} of $\bA$ with the observation
\[
D_\xi \bA (\xi) \tilde \xi \cdot \tilde \xi \ge 2^{q_1-p+2}\Lambda \frac{a_2}{t_2^{p-1}} |\xi|^{p-2}|\tilde \xi|^2= 2^{q_1-p+2}\Lambda \frac{\bphi'(|\xi|)}{|\xi|}|\tilde \xi|^2 . \qedhere
\]
\end{proof}

%%%%%%%%%%%%%%%%%%%%%%%%%%%%%%%%%%%%%%%%%%%%%%%%%%%%%%%%%%%%%%%%%%%%
\subsection{Approximation for minimizers}%\label{sect:functionalapprox}

Let $F:\Omega\times \Rn \to [0,\infty)$ satisfy Assumption~\ref{ass:F} with constants $L\ge 1$ and $1<p<q$, and $\phi\in \Phic(\Omega)$ be its growth function. 
We construct an autonomous function $\bF :\Rn\to [0,\infty)$ with growth function $\bphi$ given in \eqref{phi0}. The construction is similar to, yet more delicate than, that of 
the previous subsection.
The added difficulty comes from the fact that we need to differentiate 
$F$ twice, which makes controlling the approximation more challenging. 

Fix any small ball $B_r=B_r(x_0)\Subset \Omega$ with $r\in(0,1)$ satisfying $|B_r|\le 1$. 
For $t_1\in (0,\frac{1}{2}]$, $t_2 \ge 2$, $\bar\nu\ll 1$ and $\bar\Lambda\gg 1$ that will be chosen later, we define
\begin{equation}\label{F0def}
\bF(\xi): = 
\underbrace{\bar\nu \eta_1(|\xi|) \frac{a_1}{t_1^{p-1}} |\xi|^{p}}_{=:\bF_1(\xi)}
 + \underbrace{\eta_2(|\xi|) F(x_0,\xi) }_{=:\bF_2(\xi)}
+ \underbrace{\bar\Lambda \eta_3(|\xi|) \frac{a_2}{t_2^{p-1}} |\xi|^p}_{=:\bF_3(\xi)},
\end{equation}
where $\eta_i\in C^\infty([0,\infty))$, $i=1,2,3$, satisfy
\[
\eta_1\equiv 1\ \text{ in } \ [0,t_1), \quad \eta_1\equiv 0\ \text{ in } \ [2t_1,\infty) ,\quad \eta_1' \le 0 \quad \text{and}\quad |\eta_1'|t_1 +|\eta_1''|t_1^2 \le 10,
\]
\[
\eta_2\equiv 1\ \text{ in } \ [0,t_2), \quad \eta_1\equiv 0\ \text{ in } \ [2t_2,\infty) ,\quad \eta_2'\le 0 \quad \text{and}\quad |\eta_2'|t_2 + |\eta_2''|t_2^2 \le 10,
\]
\[
\eta_3(t):=\int_0^t \frac{h(s)}{s^2}\, ds
%\equiv 0\ \text{ in } \ [0,\tfrac{t_2}{2}), \quad \eta_3\equiv 1\ \text{ in } \ [t_2,\infty), \quad \eta_3'\ge0 \quad \text{and}\quad |\eta_3'|t_2+|\eta_3''|t_2^2 \le 10.
\]
with $h\in C^\infty([0,\infty))$ increasing, equal to $0$ on $[0,\frac12 t_2]$ and to $t_2$ on
$[\frac 34 t_2, \infty]$, and $\|h'\|_\infty \le 10$. 

Observe that $\eta_1$ and $\eta_2$ are analogous to their namesakes in the previous 
subsection, but $\eta_3$ behaves somewhat differently. The reason is that if 
$\eta_3$ is forced to be constant from $2t_2$ onward as in the previous subsection, then it is not possible to control 
derivatives up to order $2$ in an appropriate manner. 
We again note that 
\begin{equation}\label{eq:etaEst}
\eta_i(t) + t |\eta_i'(t)| + t^2 |\eta_i''(t)| \le C
\end{equation}
for $i=1, 2, 3$ and all $t\ge 0$, and that 
\begin{equation*}%\label{f0f}
\bF(\xi)=F(x_0,\xi) \quad \text{whenever }\ 2t_1\le |\xi| \le \tfrac{t_2}{2}. 
\end{equation*}

\begin{lemma}
If $\phi$ is a growth function of $F$, and $\bphi$ and $\bF$ are their 
autonomous approximations as in 
\eqref{phi0} and \eqref{F0def}, then $\bphi$ is a growth function of $\bF$. 
\end{lemma}
\begin{proof} 
To check the first condition of the definition of growth function, \eqref{Agrowth} with $A:=D_\xi \bF$ and $\phi:=\bphi$, 
we calculate the derivatives and use \eqref{eq:etaEst} as in Lemma~\ref{lem:regular}; 
the only noteworthy feature is that we use \inc{p-1} to control $F(x_0,\xi)$ and its 
derivatives when $t<t_1$: 
\[
|D_\xi\bF(\xi)| \le p\bar\nu a_1 (\tfrac {|\xi|}{t_1})^{p-1}+ \Lambda \phi'(x_0,|\xi|) \le (p\bar\nu +\Lambda)a_1(\tfrac {|\xi|}{t_1})^{p-1} \approx \bphi'(|\xi|)
\]
and
\[
|\xi|\,|D_\xi^2\bF(\xi)| \le c(n,p)\bar\nu a_1 (\tfrac {|\xi|}{t_1})^{p-1} + \Lambda \phi'(x_0,|\xi|) \le (c(n,p) \bar\nu +\Lambda) a_1(\tfrac {|\xi|}{t_1})^{p-1} \approx \bphi'(|\xi|).
\]

We move on to \eqref{Aellipticity} and
consider four cases: $0<|\xi|\le t_1$, $t_1<|\xi|\le 2 t_1$, $t_2/2<|\xi|\le t_2$ 
and $|\xi|>t_2$. The remaining case $2t_1<|\xi|\le t_2/2$ follows 
from the assumption that $\phi$ is a growth function of $F$, 
since $\phi\equiv \bphi$ and $F\equiv \bF$ in this set.
 
\framebox{$0<|\xi|\le t_1$} 
As in Lemma~\ref{lem:regular}, we find that 
\[\begin{split}
D_\xi^2\bF(\xi) \tilde \xi\cdot \tilde \xi & 
= \bar\nu \frac{a_1}{t_1^{p-1}}D_\xi(|\xi|^{p-2}\xi)\tilde \xi \cdot \tilde \xi + 
\underbrace{D_\xi^2F(x_0,\xi) \tilde \xi\cdot \tilde \xi}_{\ge 0} \\
% & \ge p\bar\nu \frac{a_1}{t_1^{p-1}}D(|\xi|^{p-2}\xi)\tilde \xi \cdot \tilde \xi \\
& \ge \bar\nu \min\{p-1,1\} \frac{a_1}{t_1^{p-1}} |\xi|^{p-2}|\tilde \xi|^2 
= \bar\nu \min\{p-1,1\}\frac{\bphi'(|\xi|)}{|\xi|} |\tilde \xi|^2.
\end{split}\]

\framebox{$t_1<|\xi|\le 2 t_1$} 
We calculate
\[\begin{split}
D_\xi \bF(\xi) 
&= \bar\nu \big( \eta_1'(|\xi|)|\xi|+p\eta_1(|\xi|)\big) \frac{a_1}{t_1^{p-1}} |\xi|^{p-2}\xi + D_\xi F(x_0,\xi)
\end{split}\]
and
\[\begin{split}
D_\xi^2\bF(\xi) 
& = \bar\nu\big[ \eta_1''(|\xi|)|\xi|^2+ (1+p)\eta_1'(|\xi|)|\xi|\big] \frac{a_1}{t_1^{p-1}} |\xi|^{p-4} \xi\otimes \xi\\
&\qquad + \bar\nu \big[ \eta_1'(|\xi|)|\xi|+p\eta_1(|\xi|)\big] \frac{a_1}{t_1^{p-1}} D_\xi (|\xi|^{p-2}\xi) + D_\xi^2 F(x_0,\xi).
\end{split}\]
By \eqref{eq:etaEst}, the coefficients in the square brackets are bounded from 
below by a negative constant. 
By taking $\bar\nu$ sufficiently small and using \inc{p-1} and \dec{q-1} of $\bphi'$, we have 
\[\begin{split}
D_\xi^2\bF(\xi) \tilde \xi\cdot \tilde \xi & \ge - \bar\nu c \frac{a_1}{t_1^{p-1}} |\xi|^{p-2} |\tilde \xi|^2 + \nu \frac{\phi'(x_0,|\xi|)}{|\xi|}|\tilde \xi|^2 
 \ge 
%- \frac{\nu}{2} \frac{\phi'(x_0,|\xi|)}{|\xi|}|\tilde \xi|^2 + \nu \frac{\phi'(x_0,|\xi|)}{|\xi|}|\tilde \xi|^2 
%= 
\frac{\nu}{2} \frac{\phi'(x_0, |\xi|)}{|\xi|}|\tilde \xi|^2  
\approx
\frac{\nu}{2} \frac{\bphi'(|\xi|)}{|\xi|}|\tilde \xi|^2  .
\end{split}\]

\framebox{$\frac{t_2}{2}<|\xi|\le t_2$}
We observe that 
$\eta_3'(t)=\frac{h(t)}{t^2}$. Since $h(t)=\eta_3'(t)t^2$ is increasing and differentiable, 
we conclude that $(\eta_3'(t)t^2)'=\eta_3''(t)t^2 + 2t\eta_3'(t)\ge 0$. 
Also $p\eta_3(t)+ \eta_3'(t)t\ge 0$ since both terms are non-negative. We calculate
$
D_\xi \bF_3(\xi)= \bar\Lambda \big( p\eta_3(|\xi|)+ \eta_3'(|\xi|) |\xi|\big) \frac{a_2}{t_2^{p-1}} |\xi|^{p-2}\xi 
$
and 
\begin{align*}
D_\xi^2\bF_3(\xi) \tilde \xi \cdot \tilde \xi 
& = \bar\Lambda \big( (p+1)\eta_3'(|\xi|)|\xi|+ \eta_3''(|\xi|)|\xi|^2 \big)\frac{a_2}{t_2^{p-1}} |\xi|^{p-4} 
(\xi\cdot \tilde \xi)^2\\
&\qquad+ \bar\Lambda \big( p\eta_3(|\xi|)+ \eta_3'(|\xi|)|\xi|\big) \frac{a_2}{t_2^{p-1}}D_\xi(|\xi|^{p-2}\xi) \tilde \xi \cdot \tilde \xi
\ge 0.
\end{align*}
Therefore, the ellipticity of $\bF$ follows from the ellipticity of $\bF_2(\xi)=A(x_0,\xi)$. 

\framebox{$t_2<|\xi|$} 
As in the previous case, we conclude that 
\[
D_\xi^2\bF_3(\xi)\tilde \xi\cdot \tilde \xi 
\ge 
\bar\Lambda  p\eta_3(|\xi|) \frac{a_2}{t_2^{p-1}} D_\xi(|\xi|^{p-2}\xi) \tilde \xi \cdot \tilde \xi
\ge 
\tfrac{\bar\Lambda p}3 \min\{1,p-1\} \frac{a_2}{t_2^{p-1}}|\xi|^{p-2}|\tilde \xi|^2,
\]
where we used that $\eta_3'\ge 0$ and 
\[
\eta_3(|\xi|) \ge \eta_3(t_2) 
= \int_0^{t_2} \frac{h(s)}{s^2} \, ds 
\ge t_2\int_{\frac34 t_2}^{t_2} \frac{1}{s^2} \, ds 
= \frac 13. 
\]
Hence the ellipticity of $\bF=\bF_3$ when $|\xi|>2t_2$ follows. 

Finally, we suppose that $t_2<|\xi|\le 2 t_2$. We calculate
$D_\xi\bF_2(\xi)= \frac{\eta_2'(|\xi|)}{|\xi|} F(x_0,\xi) \xi + \eta_2(|\xi|)D_\xi F(x_0,\xi)$
and 
\[\begin{split}
D_\xi^2\bF_2(\xi) = & \frac{\eta_2'(|\xi|)}{|\xi|} F(x_0,\xi) I_n + 2\frac{\eta_2'(|\xi|)}{|\xi|} D_\xi F(x_0,\xi) \otimes \xi + \frac{\eta_2''(|\xi|)|\xi|-\eta_2'(|\xi|)}{|\xi|^3} F(x_0,\xi) \xi\otimes \xi\\
&\quad + \eta_2(|\xi|)D_\xi^2F(x_0,\xi) .
\end{split}\]
Since $|\eta_2'|t_2+|\eta_2''|t_2^2 \le 10$ and $|\xi|\approx t_2$, by choosing 
$\bar\Lambda$ sufficiently large for the second estimate, we obtain that 
\[\begin{split}
D_\xi^2\bF(\xi) \tilde \xi \cdot \tilde \xi 
& \ge - c \frac{\phi'(x_0,|\xi|)}{|\xi|} |\tilde \xi|^2 + \tfrac{\bar\Lambda p}{3} \min\{p-1,1\} \frac{a_2}{t_2^{p-1}}|\xi|^{p-2} |\tilde \xi|^2
 \gtrsim \frac{\bphi'(|\xi|)}{|\xi|^{p-1}} |\tilde \xi|^2;
\end{split}\]
here we also used \inc{p-1} and \dec{q_1-1} of $\phi'$ to conclude that 
\[
\frac{\phi'(x_0,|\xi|)}{|\xi|^{p-1}} 
\le 
\frac{\phi'(x_0,2t_2)}{(2t_2)^{p-1}} 
\le 
2^{q_1-p}\frac{\phi'(x_0,t_2)}{t_2^{p-1}}
=
2^{q_1-p}\frac{\bphi'(|\xi|)}{|\xi|^{p-1}}.
\]
Thus we obtain the second condition of being a growth function, \eqref{Aellipticity}. 
\end{proof}

%%%%%%%%%%%%%%%%%%%%%%%%%%%%%%%%%%%%%%%%%%%%%%%%%%%%%%%%%%%%%%%%%%%%%%%%
%%%%%%%%%%%%%%%%%%%%%%%%%%%%%%%%%%%%%%%%%%%%%%%%%%%%%%%%%%%%%%%%%%%%%%%%
%%%%%%%%%%%%%%%%%%%%%%%%%%%%%%%%%%%%%%%%%%%%%%%%%%%%%%%%%%%%%%%%%%%%%%%%

\section{Regularity of weak solutions and minimizers}\label{sect:regularity}

In this section we prove the main theorems stated in the introduction. 
We consider either $A:\Omega\times \Rn\to \Rn$ satisfying Assumption~\ref{ass:A} with $A^{(-1)}$ 
satisfying \wVA{} or $F:\Omega\times\Rn\to \R$ satisfying Assumption~\ref{ass:F} and \wVA{}. 
Note that the parameters $p$, $q$ and $L$ are from Assumption~\ref{ass:A} or~\ref{ass:F}, 
and the parameter $\bL$ is given by $\bL_K$ in \wVA{} with
\begin{equation}\label{Kchoice}
K=2c_1  \ \ \text{in the equation case}\ \ \text{or}\ \ K=2c_2  \ \ \text{in the functional case,}
\end{equation}
where $c_1$ and $c_2$ are from Propositions~\ref{prop:AphiVA} and \ref{prop:fphiVA}.
Let $\phi$ be a growth function of $A$ or $F$. Note that by Proposition~\ref{prop:AphiVA} or ~\ref{prop:fphiVA}, $A^{(-1)}$ or $F$ satisfies \aone{} when $K=1$ with $\bL_1$ depending only on $n$, $p$, $q$, $L$ and $\bL$, hence so does $\phi$.

Let $u\in W^{1,\phi}_{\loc}(\Omega)$ be a local weak solution  to \eqref{mainPDE} or a local minimizer of \eqref{mainfunctional}. Then,  
by Theorem~\ref{thm:holder} along with 
Remark~\ref{rmk:holder_functional},
%Remark~\ref{rmk:phipsi},
 we have  
$\phi(\cdot,|Du|)\in L^{1+\sigma}_{\loc}(\Omega)$ for some $\sigma=\sigma(n,p,q,L,\bL) \in (0,1)$.
With this $\sigma$, we fix the parameters $\epsilon$ and $\omega=\omega_\epsilon$ from \wVA{} by
\begin{equation}\label{epsilon}
\epsilon := \frac{\sigma}{2(2+\sigma)}<\frac{1}{6},
\end{equation}

We also fix $\Omega'\Subset\Omega$, and consider any $B_{2r}\subset \Omega'$ with $r\in(0,\frac{1}{2})$ satisfying
\[
\omega(r)\leq \frac{1}{2^{q_1}L},
\ \ 
|B_{2r}|
\leq 
\max\left\{2^{p}L, 2^{\frac{1}{1-\epsilon}}, 2^{\frac{2(1+\sigma)}{\sigma}} 
\|\phi(\cdot,|Du|)^{1+\sigma}\|_{L^1(\Omega')}^{\frac{2+\sigma}{\sigma}}\right\}^{-1}.
%\le\frac{1}{2},
\]
Note that using H\"older's inequality we have
\begin{equation}\label{Dusigmale1}
\int_{B_{2r}}\phi(x,|Du|)^{1+\frac{\sigma}{2}}\,dx
\le |B_{2r}|\left(\fint_{B_{2r}}\phi(x,|Du|)^{1+\sigma}\,dx\right)^{\frac{1+\sigma/2}{1+\sigma}} 
\le \frac{1}{2},
\end{equation}
so that
\begin{equation}\label{Dusigmale2}
\int_{B_{2r}}\phi(x,|Du|)\,dx\leq \int_{B_{2r}}\phi(x,|Du|)^{1+\frac{\sigma}{2}}\,dx+ |B_{2r}|\leq \frac{1}{2}+\frac{1}{2} = 1.
\end{equation}
Note that this allows us to take advantage of the higher integrability estimate in Theorem~\ref{thm:holder}. Next we set 
\[
t_1:= (\phi^-_{B_r})^{-1}(\omega(r)) \le \frac{1}{2}\quad \text{and}\quad t_2:=(\phi^-_{B_r})^{-1}(|B_{r}|^{-1})\ge 2. 
\]
With these $t_1$ and $t_2$, we construct $\bA$, $\bF$ and $\bphi$ as described in 
Section~\ref{sect:approx}.

%%%%%%%%%%%%%%%%%%%%%%%%%%%%%%%%%%%%%%%%%%%%%%%%%%%%%%%%%%%%%%%%%%%%%%%%%%%%%%%%%%%%%
\subsection{Regularity for weak solutions} \label{sec:PDEproof}
We first prove Theorem~\ref{thm:PDE}. We assume $A:\Omega\times \Rn\to\Rn$ satisfies Assumptions~\ref{ass:A} and $A^{(-1)}$ satisfies \wVA{}, and consider a local weak solution $u\in W^{1,\phi}_{\loc}(\Omega)$ to \eqref{mainPDE}. 
By Proposition~\ref{prop:AphiVA}(1), 
\begin{equation}\label{wva1aphi}
|A(x,\xi)-A(y,\xi)|\lesssim \omega(r)^\frac{1}{p'}\big((\phi')^-_{B_r}(|\xi|)+1\big), 
\end{equation}
for all $\xi\in\R^n$ satisfying $\phi_{B_r}^-(|\xi|)\in \left[0,|B_r|^{-1+\epsilon}\right]$.
Moreover, by Proposition~\ref{prop:AphiVA}(2), the condition of Proposition~\ref{prop:phi0} holds with $\tilde L=\tilde L(n,p,q,L,\bL)$. 

Recall that $\bphi$ and $\bA$ were constructed in \eqref{phi0} and \eqref{A0def}. Let $\bu\in u+ W^{1,\bphi}_0(B_r)$ be a weak solution to 
%the $\bA$-equation 
\eqref{eqA0}. 
By Proposition~\ref{prop:phi0}(3), $u\in W^{1,\bphi}(B_r)$, so it is a suitable 
boundary value.
The following lemma is a generalization of \cite[Lemma~5.15]{HasO22} to the non-Uhlenbeck case.
Note that 
we use only the \aone{} condition of $A^{(-1)}$ in the proof.

\begin{lemma}\label{lem:gradientestimates} 
Let $A$, $\phi$, $u$, $\bL$, $\sigma$, $B_{2r}$ and $\bu$ be as above. %Assume that $\intA$ satisfies \wVA{}.
Then
%%% FAKE ENUMERATE ENVIRONMENT
\begin{align} \tag{1}
\fint_{B_{r}}\phi(x,|Du|)\,dx 
& \le \bigg(\fint_{B_{r}}\phi(x,|Du|)^{1+\sigma}\,dx\bigg)^{\frac{1}{1+\sigma}} \\
\notag
&\lesssim
\phi^-_{B_{2r}} \bigg(\fint_{B_{2r}}|Du|\,dx\bigg)+1
\lesssim 
\bphi \bigg(\fint_{B_{2r}}|Du|\,dx\bigg)+1,
\end{align}
\begin{align} \tag{2}
\fint_{B_r}\phi(x,|D\bu|)\,dx 
&\leq \left(\fint_{B_r}\phi(x,|D\bu|)^{1+\frac{\sigma}{2}}\,dx\right)^{\frac{2}{2+\sigma}}
\lesssim 
\left(\fint_{B_r}\phi(x,|Du|)^{1+\frac{\sigma}{2}}\,dx+1\right)^{\frac{2}{2+\sigma}},
\end{align}
\begin{align} \tag{3}
\fint_{B_r}|D\bu|\,dx 
\lesssim 
\fint_{B_{2r}}|Du|\,dx+1.
\end{align}
%\begin{align} \tag{1}
%\fint_{B_{r}}\phi(x,|Du|)\,dx 
%& \le \bigg(\fint_{B_{r}}\phi(x,|Du|)^{1+\sigma}\,dx\bigg)^{\frac{1}{1+\sigma}} \\
%\notag
%&\le c\, \bigg\{
%\phi^-_{B_{2r}} \bigg(\fint_{B_{2r}}|Du|\,dx\bigg)+1\bigg\} \le c\, \bigg\{
%\bphi \bigg(\fint_{B_{2r}}|Du|\,dx\bigg)+1\bigg\},
%\end{align}
%\begin{align} \tag{2}
%\fint_{B_r}\phi(x,|D\bu|)\,dx 
%&\leq \left(\fint_{B_r}\phi(x,|D\bu|)^{1+\frac{\sigma}{2}}\,dx\right)^{\frac{2}{2+\sigma}}
%\le 
%c \left(\fint_{B_r}\phi(x,|Du|)^{1+\frac{\sigma}{2}}\,dx+1\right)^{\frac{1}{1+\sigma}},
%\end{align}
%\begin{align} \tag{3}
%\fint_{B_r}|D\bu|\,dx 
%\le 
%c\left(\fint_{B_{2r}}|Du|\,dx+1\right).
%\end{align}
The implicit constants depend only on $n$, $p$, $q$, $L$ and $\bL$.
\end{lemma}

\begin{proof}
By H\"older's inequality and higher integrability in Theorem~\ref{thm:holder}, we have
\[
\fint_{B_{r}}\phi(x,|Du|)\,dx 
\le 
\left(\fint_{B_{r}}\phi(x,|Du|)^{1+\sigma}\,dx\right)^{\frac{1}{1+\sigma}} 
\lesssim 
\phi_{B_{2r}}^- \left(\fint_{B_{2r}}|Du|\,dx\right)+1.
\]
From \eqref{Dusigmale2} we obtain that
$\fint_{B_{2r}}|Du|\,dx\lesssim (\phi^-_{B_r})^{-1} (|B_r|^{-1}) =t_2$. Hence (1) 
follows by Proposition~\ref{prop:phi0}(2). %and \azero{}. 

By Proposition~\ref{prop:phi0}(4), we see that $\theta(x,t):=\phi(x,\bphi^{-1}(t))^{1+\sigma/2}$ satisfies \azero, \ainc{1+\sigma/2} and \adec{(1+\sigma/2)q_1/p}. Moreover, \aone{} of $\phi$ implies \aone{} of $\theta$. Therefore, in view of Lemma~\ref{lem:CZ} with \eqref{Dusigmale1} we have (2). 

Finally, by Jensen's inequality, %the quasiminimizing property 
the standard energy estimate for \eqref{eqA0}, i.e.\ the estimate in Lemma~\ref{lem:CZ} with $\theta(x, t)\equiv t$, and (1), we have
\[\begin{split}
\bphi\left(\fint_{B_r}|D\bu|\,dx\right) & \le \fint_{B_r}\bphi(|D\bu|)\,dx %\textcolor{Red}{+1} 
\lesssim \fint_{B_r}\bphi(|Du|)\,dx \lesssim \fint_{B_r}\phi(x,|Du|)\,dx+1\\
& \lesssim \bphi\left(\fint_{B_r}|Du|\,dx+1 \right),
\end{split}\]
which implies (3) since $\bphi$ is strictly increasing.
\end{proof}

We next estimate the difference of the gradient between $u$ and $\bu$ in the $L^1$-sense. 
This generalizes Lemma~6.2 and Corollary~6.3 of \cite{HasO22} to the non-Uhlenbeck case.

\begin{lemma}\label{lem:comparison}
Let $A$, $\phi$, $u$, $\bL$, $\epsilon$, $\omega$, $B_{2r}$, $\bphi$, $\bA$ and $\bu$ be as above. %Assume that $\intA$ satisfies \wVA{}.
Then
\[
\fint_{B_r} \frac{\bphi'(|Du|+|D\bu|)}{|Du|+|D\bu|}|Du-D\bu|^2\,dx 
\le c\big(\omega(r)^\frac{p-1}{q_1} + r^{\gamma}\big)\left(\bphi\left(\fint_{B_{2r}}|Du|\,dx \right)+1\right)
\]
and
\begin{equation*}%\label{DuD\buL1}
\fint_{B_r} |Du-D\bu|\,dx \le c\big(\omega(r)^\frac{p-1}{2q_1^2} + r^{\frac{\gamma}{2q_1}}\big)\left(\fint_{B_{2r}}|Du|\,dx +1\right).
\end{equation*}
for some $c\ge1$ and $\gamma\in(0,1)$ depending only on $n$, $p$, $q$, $L$ and $\bL$.
\end{lemma}

\begin{proof}
Using the monotonicity of $\bA$ \eqref{monotonicity} and taking $u-\bu\in W^{1,\phi}_0(B_r)\subset W^{1,\bphi}_0(B_r)$ as a test function in the weak forms of \eqref{mainPDE} and \eqref{eqA0}, we find that
\[\begin{split}
\fint_{B_r} \frac{\bphi'(|Du|+|D\bu|)}{|Du|+|D\bu|}|Du-D\bu|^2\,dx &\lesssim \fint_{B_r} ( \bA(D\bu)-\bA(Du) )\cdot(Du-D\bu)\,dx\\
&= \fint_{B_r} ( A(x,Du)-\bA(Du) )\cdot(Du-D\bu)\,dx\\
&\le \fint_{B_r} |A(x,Du)-\bA(Du)|\,|Du-D\bu|\,dx.
\end{split}\]

We split $B_r$ into two regions defined by 
\[
\begin{split}
%E_1&:=B_r \cap\{\phi^-_{B_r}(|Du|)\leq 2^{q_1}K \omega(r)\},\\
E_1:=B_r \cap\Big\{ \phi^-_{B_r}(|Du|)\leq \frac{|B_{r}|^{-1+\epsilon}}{2^{p}}\Big\}
\quad\text{and}\quad
E_2:=B_r \cap\Big\{\frac{|B_{r}|^{-1+\epsilon}}{2^p }<\phi^-_{B_r}(|Du|)\Big\}.
\end{split}
\]
In the set $E_1$, we have $|Du|\le \frac{t_2}{2}$. If also $2t_1 \le |Du|$, then it 
follows from $\bA(\xi)=A(x_0, \xi)$ and \eqref{wva1aphi} that 
\begin{align*}
|A(x,Du)- \bA(Du)|=|A(x,Du)- A(x_0,Du)|
\lesssim 
\omega(r)^{1-\frac1p} ((\phi')^-(|Du|)+1). 
\end{align*}
If, on the other hand, $|Du|\le 2t_1$, then \dec{q_1} and \azero{} of $\phi$ imply that 
$|Du|\lesssim \omega(r)^\frac1{q_1}$. 
Therefore by \inc{p-1} and \azero{} of $\phi'$ and $\bphi'$ and the growth function property 
\eqref{Agrowth} of $A$ and $\bA$, 
we see that $|A(x,Du)- \bA(Du)|\lesssim \omega(r)^\frac{p-1}{q_1}$.
Applying Young's inequality with $\phi^-_{B_r}$ and $(\phi^-_{B_r})^*$ and  
$\phi^*(x,\phi'(x,t))\le \phi(x,t)$ as well as $t\lesssim \phi(x,t)+1$, 
we find that 
\[\begin{split}
|A(x,Du)-\bA(Du)|\,|Du-D\bu| \chi_{E_1}
&\lesssim 
\omega(r)^\frac{p-1}{q_1} [(\phi')^-(|Du|)|Du-D\bu| + |Du-D\bu|]\\
&\lesssim 
\omega(r)^\frac{p-1}{q_1} [\phi(x,|Du|)+ \phi(x,|D\bu|)+1].
\end{split}\]
Next we integrate this inequality over $B_r$ and use Lemma~\ref{lem:gradientestimates}(1)\&(2):
\[\begin{split}
\fint_{B_r}|A(x,Du)-\bA(Du)|&|Du-D\bu| \chi_{E_1}\,dx
\lesssim 
\omega(r)^\frac{p-1}{q_1} \left(\bphi\left(\fint_{B_{2r}}|Du|\,dx\right)+1\right). 
\end{split}\]

In the set $E_2$, $1\lesssim |B_{r}|^{1-\epsilon}\phi^-_{B_r}(|Du|)$. Then applying 
Young's inequality and Proposition~\ref{prop:phi0}(3) we have 
\[\begin{split}
|A(x,Du)-\bA(Du)|\,|Du-D\bu| \chi_{E_2}
&\lesssim
\phi'(x,|Du|)|Du-D\bu|
\lesssim
\phi(x,|Du|) + \phi(x,|D\bu|)\\
&\lesssim
\left[|B_{2r}|^{1-\epsilon}\phi^-_{B_r}(|Du|)\right]^{\frac{\sigma}{2}} [\phi(x,|Du|) + \phi(x,|D\bu|)]\\
&\lesssim
r^{\frac{n(1-\epsilon)\sigma}{2}}\big( \phi(x,|Du|)^{1+\frac{\sigma}{2}}+\phi(x,|D\bu|)^{1+\frac{\sigma}{2}}\big). 
\end{split}\]
Integrating this inequality over $B_r$ and using Lemma~\ref{lem:gradientestimates}(2) and the definition of 
$\epsilon$ from \eqref{epsilon}, we find that 
\[
\fint_{B_r}|A(x,Du)- \bA(Du)| |Du-D\bu| \chi_{E_2}\,dx
\lesssim
r^{\frac{n(4+\sigma)\sigma}{4(2+\sigma)}} 
\left(\fint_{B_r}\phi(x,|Du|)^{1+\frac{\sigma}{2}}\,dx\right)^{\frac{\sigma}{2+\sigma}+\frac{2}{2+\sigma}}.
\]
On one hand, by Lemma~\ref{lem:gradientestimates}(1), we have
\[
\left(\fint_{B_r}\phi(x,|Du|)^{1+\frac{\sigma}{2}}\,dx\right)^{\frac{2}{2+\sigma}}
\lesssim \bphi\left( \fint_{B_{2r}}|Du|\,dx\right)+1.
\]
On the other hand, by \eqref{Dusigmale1},
\[
\left(\fint_{B_r}\phi(x,|Du|)^{1+\frac{\sigma}{2}}\,dx\right)^{\frac{\sigma}{2+\sigma}}\le 
 |B_r|^{-\frac{\sigma}{2+\sigma}} \lesssim r^{-\frac{4n\sigma}{4(2+\sigma)}}.
\]
Therefore, the previous three inequalities imply that
\[
\fint_{B_r} |A(x,Du)-\bA(Du)|\,|Du-D\bu| \chi_{E_2}\,dx
 \lesssim 
r^{\frac{n\sigma^2}{4(2+\sigma)}}
\left(\bphi\left(\fint_{B_{2r}}|Du|\,dx\right)+1\right).
\]
Combining the results of this and the previous paragraph, 
we have the first claim of the lemma, 
with $\gamma:=\frac{n\sigma^2}{4(2+\sigma)}$. 

Next, set $\omega_0(r):=\omega(r)^\frac{p-1}{q_1}+r^\gamma $.
Applying Proposition~\ref{prop000}(3) with 
$\kappa=\omega_0(r)^{\frac{1}{2}}$, Proposition~\ref{prop:phi0}(3), Lemma~\ref{lem:gradientestimates}(1)\&(2) and the first part of the lemma, we find that
\[\begin{split}
\fint_{B_r}&\bphi(|Du-D\bu|)\,dx \\
&\lesssim \omega_0(r)^\frac{1}{2} \fint_{B_r}[\bphi(|Du|)+\bphi(|D\bu|)] \,dx 
+ \omega_0(r)^{-\frac{1}{2}}\fint_{B_r}\frac{\bphi'(|Du|+|D\bu|)}{|Du|+|D\bu|}|Du-D\bu|^2\,dx \\
&\lesssim 
\omega_0(r)^\frac{1}{2} \fint_{B_r}[\phi(x,|Du|) + \phi(x,|D\bu|)+1]\,dx
+ \omega_0(r)^\frac{1}{2}\left(\bphi\left(\fint_{B_{2r}}|Du|\,dx \right)+1\right)\\
&\lesssim \omega_0(r)^\frac{1}{2}\left(\bphi\left(\fint_{B_{2r}}|Du|\,dx \right)+1\right).
\end{split}\]
Therefore, by Jensen's inequality and \dec{q_1} of $\bphi$, we have 
\[
\bphi\left(\fint_{B_r}|Du-D\bu|\, dx\right)
\le \fint_{B_r}\bphi(|Du-D\bu|)\, dx
\lesssim \bphi\left(\omega_0(r)^{\frac{1}{2q_1}}\left(\fint_{B_{2r}}|Du|\,dx+1\right)\right).
\]
Since $\bphi$ is strictly increasing, this implies the second claim of the lemma.
\end{proof}

Now, the regularity result for weak solutions follows.

\begin{proof}[Proof of Theorem~\ref{thm:PDE}] 
We can prove the theorem using Lemmas~\ref{lem:holder}, \ref{lem:gradientestimates} and \ref{lem:comparison}  using a standard iteration argument. The proof is exactly the same as \cite[Theorem 7.2]{HasO22}, see also \cite{AM1,BarCM18}, so we omit it. 
\end{proof}

%%%%%%%%%%%%%%%%%%%%%%%%%%%%%%%%%%%%%%%%%%%%%%%%%%%%%%%%%%%%%%%%%%%%
\subsection{Regularity for minimizers}%\label{sec:functionalproof}

In this subsection we prove Theorem~\ref{thm:functional}. The method is almost the same as for
Theorem~\ref{thm:PDE} except for the comparison step. Hence we will take advantage of 
many parts from the previous subsection. 

We assume $F:\Omega\times \Rn\to\R$ satisfies Assumption~\ref{ass:F} and \wVA{}, and  consider a local minimizer $u\in W^{1,\phi}_{\loc}(\Omega)$ of \eqref{mainfunctional}. Recall that the minimizer $u$ is also a weak solution to \eqref{mainPDE} with $A=D_\xi F$, but 
\wVA{} of $F$ does not imply \wVA{} of $A:=D_\xi F$. On the other hand, by Proposition~\ref{prop:fphiVA}(1) we have
%implies  for any $x,y\in B_r\cap\Omega$ with $r\in (0,1]$ that 
\begin{equation}\label{eq:VA1functional}
F^+_{B_r}(\xi)-F^-_{B_r}(\xi) \lesssim \omega(r)\big(\phi^-_{B_r}(|\xi|)+1\big), 
\end{equation}
for all $\xi\in\R^n$ satisfying 
$\phi_{B_r}^-(|\xi|)\in \left[0,|B_r|^{-1+\epsilon}\right]$.
Moreover, by the Proposition~\ref{prop:fphiVA}(2), the condition of Proposition~\ref{prop:phi0} holds with $\tilde L=\tilde L(n,p,q,L,\bL)$. 

We use $\bphi$ and $\bF$ constructed in \eqref{phi0} and \eqref{F0def}. Let $\bu\in u+ W^{1,\bphi}_0(B_r)$ be a weak solution to \eqref{eqA0} with $\bA:=D_\xi \bF$.
Lemma~\ref{lem:gradientestimates} holds also for the minimizer $u$, 
since the lemma needed only \aone{}, which holds by Propositions~\ref{prop:Aphiaone} and~\ref{prop:fphiVA}.
We prove the following analogue of Lemma~\ref{lem:comparison}, where 
\wVA{} is used in a different way.

\begin{lemma}\label{lem:comparisonfunctional}
Let $F$, $\phi$, $u$, $\bL$, $\epsilon$, $\omega$, $B_{2r}$, $\bphi$, $\bF$ and $\bu$ be as above.
Then
\[
\fint_{B_r} \frac{\bphi'(|Du|+|D\bu|)}{|Du|+|D\bu|}|Du-D\bu|^2\,dx \le c\big(\omega(r)^\frac{p}{q_1} + r^{\gamma}\big)\left(\bphi\left(\fint_{B_{2r}}|Du|\,dx \right)+1\right)
\]
and
\[
\fint_{B_r} |Du-D\bu|\,dx \le c\big(\omega(r)^\frac{p}{2q_1^2} + r^{\frac{\gamma}{2q_1}}\big)
\left(\fint_{B_{2r}}|Du|\,dx +1\right)
\]
for some $c\ge1$ and $\gamma\in(0,1)$ depending only on $n$, $p$, $q$, $L$ and $\bL$.
\end{lemma} 

\begin{proof} 
By monotonicity \eqref{monotonicity} and \inc{p-1} and \dec{q_1-1} of $\bphi'$, 
\[\begin{split}
\bF(\xi_1)-\bF(\xi_2) -D_\xi \bF(\xi_2) \cdot (\xi_1-\xi_2) 
&= \int_0^{1}( D_\xi \bF(t\xi_1+(1-t)\xi_2) - D_\xi \bF(\xi_2) )\cdot (\xi_1-\xi_2)\,dt \\
&\gtrsim \int_0^{1} t\frac{\bphi'(|t\xi_1+(1-t)\xi_2|+|\xi_2|)}{|t\xi_1+(1-t)\xi_2| +|\xi_2|}|\xi_1-\xi_2|^2\,dt \\
&\gtrsim \int_{3/4}^{1} t\frac{\bphi'(|t\xi_1+(1-t)\xi_2|+|\xi_2|)}{|\xi_1|+|\xi_2|}|\xi_1-\xi_2|^2\,dt\\
&\gtrsim \left( \int_{3/4}^{1} t\,dt\right) \frac{\bphi'(\frac{3}{4}|\xi_1|-\frac{1}{4}|\xi_2|+|\xi_2|)}{|\xi_1|+|\xi_2|}|\xi_1-\xi_2|^2 \\
&\approx \frac{\bphi'(|\xi_1|+|\xi_2|)}{|\xi_1|+|\xi_2|}|\xi_1-\xi_2|^2
\end{split}
\]
for every $\xi_1,\xi_2\in\Rn$.
Using this and the facts that $u-\bu\in  W^{1,\phi}_0(B_r) \subset W^{1,\bphi}_0(B_r)$ from Lemma~\ref{lem:gradientestimates}, we find that 
\[
\begin{aligned}
&\fint_{B_r}\frac{\bphi'(|Du|+|D\bu|)}{|Du|+|D\bu|}|Du-D\bu|^2\,dx \\
&\quad\lesssim 
\fint_{B_r}[\bF(Du)- \bF(D\bu)]\,dx - 
\underbrace{\fint_{B_r} D_\xi \bF(D\bu)\cdot(Du-D\bu)\,dx}_{=0 \text{ by \eqref{eqA0}}}\\
&\quad =
\fint_{B_r}[\bF(Du)- F(x,Du)+F(x,D\bu)- F(|D\bu|)]\,dx + 
\underbrace{\fint_{B_r}[F(x,Du)- F(x,D\bu)]\, dx}_{\le 0\text{ by \eqref{mainfunctional}}}\\
&\quad\leq 
\fint_{B_r}|\bF(Du)- F(x,Du)|\,dx+ 
\fint_{B_r}|F(x,D\bu)- \bF(D\bu)|\,dx.
\end{aligned}
\] 
We shall estimate only the second integral; the estimate for the first is analogous, we merely 
swap the roles of $D\bu$ and $Du$. We 
%denote by $K$ the constant from $F\approx \phi$ \eqref{phifequiv} and 
split $B_r$ into 
\[
\begin{aligned}
%E_1&:=B_r \cap\{\phi^-_{B_r}(|D\bu|)\leq 2^{q_1}K\omega(r)\},\\
E_1:=B_r \cap\left\{ \phi^-_{B_r}(|D\bu|)\leq \frac{|B_{r}|^{-1+\epsilon}}{2^p}\right\}
\quad\text{and}\quad
E_2:=B_r \cap\left\{\frac{|B_{r}|^{-1+\epsilon}}{2^p }<\phi^-_{B_r}(|D\bu|)\right\}.
\end{aligned}
\]
%In the set $E_1$, since $F\approx \phi$ and $\bF\approx \bphi$ and $\phi$ and $\bphi$ 
%satisfy \azero, \inc{p} and \dec{q_1}, 
%\[
%\fint_{B_r}|F(x,D\bu)- \bF(D\bu)| \chi_{E_1}\,dx
%\lesssim 
%\omega(r)^\frac p{q_1}. 
%\]
In the set $E_1$, $|D\bu|\leq \frac{t_2}{2}$. If also $2t_1 \le |D\bu|$, then $\bF(\xi)=F(x_0,\xi)$ and \eqref{eq:VA1functional} imply that
\[
|F(x,D\bu)- \bF(D\bu)| = |F(x,D\bu)- F(x_0,D\bu)|
\lesssim \omega(r) (\phi^-_{B_r}(|D\bu|)+1). 
\]
When $|D\bu|< 2t_1$, we use $F\approx \phi$, $\bF\approx \bphi$ and 
that $\phi$ and $\bphi$ satisfy \azero, \inc{p} and \dec{q_1}, 
to conclude that $|F(x,D\bu)- \bF(D\bu)|\lesssim \omega(r)^\frac p{q_1}$. 
Therefore, applying Lemma~\ref{lem:gradientestimates}(1)\&(2), we have 
\[
\fint_{B_r}|F(|x,D\bu|)- \bF(D\bu)| \chi_{E_1}\,dx 
\lesssim 
\omega(r)^\frac p{q_1} \left(\bphi\left(\fint_{B_{2r}}|Du|\,dx\right)+1\right).
\]

In the set $E_2$, Proposition~\ref{prop:phi0}(3), $1\lesssim |B_{2r}|^{1-\epsilon}\phi^-_{B_r}(|D\bu|)$ and $F\approx \phi$ imply that 
\[
|F(x,D\bu)- \bF(D\bu)|
\lesssim
\phi(x,|D\bu|) 
\lesssim
r^{\frac{n(1-\epsilon)\sigma}{2}} \phi(x,|D\bu|)^{1+\frac{\sigma}{2}}. 
\]
Integrating this inequality over $E_2$ and using the definition of $\epsilon$ from \eqref{epsilon}, we find that 
\[
\fint_{B_r}|F(x,D\bu)- \bF(D\bu)| \chi_{E_2}\,dx
\lesssim
r^{\frac{n(4+\sigma)\sigma}{4(2+\sigma)}} 
\left(\fint_{B_r}\phi(x,|D\bu|)^{1+\frac{\sigma}{2}}\,dx\right)^{\frac{\sigma}{2+\sigma}+\frac{2}{2+\sigma}}.
\]
In the same way as in the proof of Lemma~\ref{lem:comparison}, we derive from this that
\[
\fint_{B_r}|F(x,D\bu)- F(D\bu)| \chi_{E_2}\,dx
 \lesssim 
r^{\frac{n\sigma^2}{4(2+\sigma)}}
\left(\bphi\left(\fint_{B_{2r}}|Du|\,dx\right)+1\right).
\]
Adding the estimates in $E_1$ and $E_2$, we obtain the inequality
\[
\fint_{B_r}\frac{\bphi'(|Du|+|D\bu|)}{|Du|+|D\bu|}|Du-D\bu|^2\,dx
\lesssim
\left(\omega(r)^\frac{p}{q_1} + r^{\gamma} \right)
\left(\bphi\left(\fint_{B_{2r}}|Du|\,dx\right)+1\right)
\]
with $\gamma:=\frac{n\sigma^2}{4(2+\sigma)}$.
This is the first claim of the lemma. The second claim follows from the first one 
by the same argument as in the proof of Lemma~\ref{lem:comparison}.
\end{proof}

\begin{proof}[Proof of Theorem~\ref{thm:functional}]
The proof is exactly same as that of Theorem~\ref{thm:PDE}, with Lemma~\ref{lem:comparison} replaced by Lemma~\ref{lem:comparisonfunctional}.
\end{proof}

%%%%%%%%%%%%%%%%%%%%%%%%%%%%%%%%%%%%%%%%%%%%%%%%%%%%%%%%%%%%%%%%%%%%%%%%
%%%%%%%%%%%%%%%%%%%%%%%%%%%%%%%%%%%%%%%%%%%%%%%%%%%%%%%%%%%%%%%%%%%%%%%%
%%%%%%%%%%%%%%%%%%%%%%%%%%%%%%%%%%%%%%%%%%%%%%%%%%%%%%%%%%%%%%%%%%%%%%%%

\section{Examples} \label{sect:example}

We present two known, important nonstandard growth problems, and show that they are special cases of Theorem~\ref{thm:PDE} or~\ref{thm:functional}. Various other examples of  growth functions together with references about regularity results for related equations and minimizing problems can be found in \cite{HasO22}. 

The first example is the equation \eqref{mainPDE} with $p(x)$-growth, for which we refer to \cite{Fan1} (see also \cite{AM1}). 
The second example is  the functional \eqref{mainfunctional} with  the double phase condition, for which we refer to \cite{BarCM18} (see also \cite{ColM15-1,ColM15-2}). For brevity, we consider one example in PDE form and the other as a minimizer, although obviously both problems could be considered in either form.  In Example~\ref{example:variable}, $p$ is a function whereas the lower growth exponent is denoted by $p^-$.

\begin{example}[Variable exponent growth]\label{example:variable} Let $p:\Omega\to [p^-,p^+]$ for some $1<p^-\le p^+$, and let $\omega_p$ be the modulus of continuity of $p$ and satisfy
\[
\lim_{r\to0} \omega_p(r)\ln\left(\tfrac{1}{r}\right)=0.
\] 
Suppose $A:\Omega\times \Rn\to\Rn$ satisfies that $A(x,\cdot)\in C^1(\Rn\setminus\{0\},\R^n)$ and
\[
\begin{cases}
|A(x,\xi)|+ |\xi| |D_\xi A(x,\xi)|  \le L |\xi|^{p(x)-1},\\
|\xi|^{p(x)-2} |\tilde \xi|^2 \le L D_\xi A(x,\xi) \tilde \xi \cdot \tilde \xi , \\
|A(x,\xi)-A(y,\xi)| \le L \omega_p(|x-y|)(|\xi|^{p(x)-1}+|\xi|^{p(y)-1})\big(1+\big|\ln|\xi|\,\big|\big).
\end{cases} 
\]
For instance, $A(x,\xi)=|\xi|^{p(x)-2}\xi$ satisfies these conditions.

We will show that $A$ satisfies the assumptions of Theorem~\ref{thm:PDE}.
Since the first two conditions above imply 
\[
|\xi|^{p(x)-2} \approx |D_\xi A(x,\xi)| \approx |\xi|^{-1}|A(x,\xi)| \quad \text{and} \quad   |D_\xi A(x,\xi)|\lesssim  D_\xi A(x,\xi) e  \cdot e, 
\]
Assumption~\ref{ass:A} is satisfied. 
We next derive \wVA{} for $A^{(-1)}$. Let $r\in(0,1)$ and $x,y\in B_r$. 
%The third and first conditions 
The third condition and $|\xi|^{p(x)-2} \approx  |\xi|^{-1}|A(x,\xi)|$
imply that
\[
|\xi|\,|A(x,\xi)-A(y,\xi)|  
\lesssim
\omega_p(r)\big(1+\big|\ln|\xi|\big|\big) |\xi|\,|A(y,\xi)|
\]
for all $\xi\in \Rn$. Moreover, if $1\le |\xi|\,|A(y,\xi)|\le |B_r|^{-1}$, then 
%$|B_1|r^{-n}\ge L^{-1}|\xi|^{p^-}$ 
$|\xi|^{p^-}\le C r^{-n}$ for some $C\ge 1$ so that
\[
1+\ln |\xi| \le  1+ \tfrac1{p^-}\ln 
C
+ \tfrac{n}{p^-}\ln\left(\tfrac{1}{r}\right) \lesssim \ln\left(\tfrac{1}{r}\right).
\]
On the other hand, if $0\le |\xi| |A(y,\xi)|\le 1$, then $|\xi| |A(y,\xi)| \lesssim |\xi|^{p^-}$ hence
\[
\big(1+\big|\ln|\xi|\big|\big)\, |\xi|\,|A(y,\xi)|
\lesssim 
\big(1+\big|\ln|\xi|\big|\big)\, |\xi|^{p^-} \lesssim 1.
\]
Therefore we have 
\[
|\xi|\,|A(x,\xi)-A(y,\xi)| \lesssim \omega_p(r)\ln\left(\tfrac{1}{r}\right) (|\xi|\,|A(y,\xi)|+1), \quad |\xi|\,|A(y,\xi)|\in(0,|B_r|^{-1}].
\]
This implies $A^{(-1)}$ satisfies \wVA{} with $\omega_\epsilon(r)= \omega_p(r)\ln\left(\tfrac{1}{r}\right)$ for all $\epsilon\in(0,1]$. 

We note that if $\omega_p(r)\lesssim r^\beta$ for some $\beta>0$, then $\omega_\epsilon(r)\lesssim r^{\beta_\epsilon}$ for any $\beta_\epsilon \in(0,\beta)$.
\end{example}

\begin{example}[Double phase growth] 
Let $H(x,t)=t^p+a(x)t^q$ with $a:\Omega\to [0,a^+]$ satisfy that for some $\beta\in(0,1]$,
\[
1\le \frac{q}{p}\le 1+\frac{\beta}{n} 
%,
%\quad
%0\le a(\cdot)\le L
\quad \text{and}\quad 
|a(x)-a(y)|\le L|x-y|^\beta.
\]
Suppose $F:\Omega\times \Rn\to\R$ satisfies $F(x,\cdot)\in C^1(\Rn)\cap C^2(\Rn\setminus\{0\})$,
\[
\begin{cases}
|\xi | \, |D_\xi F(x,\xi)| + |\xi|^2 \, |D_\xi^2 F(x,\xi)| \le L H(x,|\xi|),
\\
|\xi|^{-2}H(x,|\xi|) |\tilde \xi|^2 \le L D_\xi^2F(x,\xi) \tilde \xi \cdot \tilde \xi ,
\\
 |F(x,\xi)-F(y,\xi)| \le L \tilde\omega(|x-y|)(H(x,|\xi|)+H(y,|\xi|))+ L |a(x)-a(y)|\,|\xi|^q.
\end{cases}
\]
For instance, $F(x,\xi)=\gamma(x) H(x,\xi_1^4+\cdots+\xi_n^4)$ with $\tilde\omega$ the modulus of continuity of $\gamma$  satisfies these conditions if $0<\gamma^- \le \gamma \le \gamma^+$.

We will show that $F$ satisfies the assumptions of Theorem~\ref{thm:functional}.
Assumption~\ref{ass:F} is obvious from the first two conditions above. 
We next show \wVA{}. Let $r\in(0,1)$ and $x,y\in B_r$. The first condition implies that
\[
|F(x,\xi)-F(y,\xi)| \lesssim \tilde\omega(2r)(|H(x,\xi)|+|H(y,\xi)|)+  r^\beta |\xi|^q
\]
for all $\xi\in \Rn$. From the first and second conditions, we conclude that 
$|H(x,\xi)|  \lesssim F(x,\xi)$ and $|H(y,\xi)|  \lesssim F(y,\xi)$.
If $F(y,\xi)\le |B_r|^{-1+\epsilon}$, 
then $|\xi|^p\lesssim r^{-(1-\epsilon)n}$. Since $\beta \ge \frac{(q-p)n}{p}$, we have
\[
r^\beta |\xi|^{q}\le  r^\beta|\xi|^{q-p}H(y,|\xi|) \lesssim r^{\frac{\epsilon(q-p)n}{p}} F(y,|\xi|)
\]
Therefore we have 
\[
|F(x,\xi)-F(y,\xi)| \lesssim
\left(\tilde\omega(2r)+r^{\frac{\epsilon(q-p)n}{p}}\right)(|F(x,\xi)| + |F(y,\xi)|).
\]
From this estimate we obtain $|F(x,\xi)| \lesssim |F(y,\xi)|$, and this 
we have \wVA{} with $\omega_\epsilon(r)= \tilde \omega(2r)+r^{\frac{\epsilon(q-p)n}{p}}$. 

We note that if $\tilde\omega(r) \lesssim r^\beta$ for some $\beta>0$, then $\omega_\epsilon (r)\lesssim r^{\min\{\beta, \frac{\epsilon(q-p)n}{p}\}}$.
\end{example}

%%%%%%%%%%%%%%%%%%%%%%%%%%%%%%%%%%%%%%%%%%%%%%%%%%%%%%%%%%%%%%%%%%%%%%%%
%%%%%%%%%%%%%%%%%%%%%%%%%%%%%%%%%%%%%%%%%%%%%%%%%%%%%%%%%%%%%%%%%%%%%%%%
%%%%%%%%%%%%%%%%%%%%%%%%%%%%%%%%%%%%%%%%%%%%%%%%%%%%%%%%%%%%%%%%%%%%%%%%

\section*{Acknowledgment and data statement}

We thank the referee for comments. Peter Hästö was supported in part by the 
Jenny and Antti Wihuri Foundation and Jihoon Ok was supported by the 
National Research Foundation of Korea by the Korean Government (NRF-2022R1C1C1004523).

Data sharing is not applicable to this article as obviously no datasets were generated or analyzed during the current study.

%%%%%%%%%%%%%%%%%%%%%%%%%%%%%%%%%%%%%%%%%%%%%%%%%%%%%%%%%%%%%%%%%%%%%%%%
%%%%%%%%%%%%%%%%%%%%%%%%%%%%%%%%%%%%%%%%%%%%%%%%%%%%%%%%%%%%%%%%%%%%%%%%
%%%%%%%%%%%%%%%%%%%%%%%%%%%%%%%%%%%%%%%%%%%%%%%%%%%%%%%%%%%%%%%%%%%%%%%%

\bibliographystyle{amsplain}

\end{document}